\documentclass[11pt,letterpaper]{amsart}

\usepackage{graphicx}
\usepackage{amssymb}
\usepackage{amsthm}
\usepackage{amsmath}
\usepackage{stmaryrd}
\usepackage{cite}
\usepackage[mathscr]{eucal}
\usepackage{hyperref}
\usepackage{comment}
\usepackage{caption}
\usepackage{subcaption}
\usepackage{wrapfig}
\usepackage[permil]{overpic}
\usepackage{url}
\usepackage{hyperref}
\usepackage{placeins}

\title[Einstein metrics on $S^{10}$]{Einstein metrics on the ten-sphere}
\author[Nienhaus and Wink]{Jan Nienhaus and Matthias Wink}
\address{Department of Mathematics, UCLA, 520 Portola Plaza, Los Angeles, CA, 90095}
\email{nienhaus@math.ucla.edu}

\address{Dpartment of Mathematics, University of California, Santa Barbara, South Hall, Room 6607, Santa Barbara, CA 93106, USA}
\email{wink@math.ucsb.edu}


\makeatletter
\@namedef{subjclassname@2020}{
\textup{2020} Mathematics Subject Classification}
\makeatother


\begin{document}
\newcommand{\Ext}{\bigwedge\nolimits}
\newcommand{\Div}{\operatorname{div}}
\newcommand{\Hol} {\operatorname{Hol}}
\newcommand{\diam} {\operatorname{diam}}
\newcommand{\Scal} {\operatorname{Scal}}
\newcommand{\scal} {\operatorname{scal}}
\newcommand{\Ric} {\operatorname{Ric}}
\newcommand{\Hess} {\operatorname{Hess}}
\newcommand{\grad} {\operatorname{grad}}
\newcommand{\Sect} {\operatorname{Sect}}
\newcommand{\Rm} {\operatorname{Rm}}
\newcommand{ \Rmzero } {\mathring{\Rm}}
\newcommand{\Rc} {\operatorname{Rc}}
\newcommand{\Curv} {S_{B}^{2}\left( \mathfrak{so}(n) \right) }
\newcommand{ \tr } {\operatorname{tr}}
\newcommand{ \id } {\operatorname{id}}
\newcommand{ \Riczero } {\mathring{\Ric}}
\newcommand{ \ad } {\operatorname{ad}}
\newcommand{ \Ad } {\operatorname{Ad}}
\newcommand{ \dist } {\operatorname{dist}}
\newcommand{ \rank } {\operatorname{rank}}
\newcommand{\Vol}{\operatorname{Vol}}
\newcommand{\dVol}{\operatorname{dVol}}
\newcommand{ \zitieren }[1]{ \hspace{-3mm} \cite{#1}}
\newcommand{ \pr }{\operatorname{pr}}
\newcommand{\diag}{\operatorname{diag}}
\newcommand{\Lagr}{\mathcal{L}}
\newcommand{\av}{\operatorname{av}}
\newcommand{ \floor }[1]{ \lfloor #1 \rfloor }
\newcommand{ \ceil }[1]{ \lceil #1 \rceil }
\newcommand{\Sym} {\operatorname{Sym}}
\newcommand{\bcirc}{ \ \bar{\circ} \ }
\newcommand{\conj}[1]{ \overline{ #1 } }
\newcommand{\sign}[1]{\operatorname{sign}(#1)}
\newcommand{\cone}{\operatorname{cone}}
\newcommand{\pbd}{\varphi_{bar}^{\delta}}

\newtheorem{theorem}{Theorem}[section]
\newtheorem{definition}[theorem]{Definition}
\newtheorem{example}[theorem]{Example}
\newtheorem{remark}[theorem]{Remark}
\newtheorem{lemma}[theorem]{Lemma}
\newtheorem{proposition}[theorem]{Proposition}
\newtheorem{corollary}[theorem]{Corollary}
\newtheorem{assumption}[theorem]{Assumption}
\newtheorem{acknowledgment}[theorem]{Acknowledgment}
\newtheorem{DefAndLemma}[theorem]{Definition and lemma}

\newcommand{\R}{\mathbb{R}}
\newcommand{\N}{\mathbb{N}}
\newcommand{\Z}{\mathbb{Z}}
\newcommand{\Q}{\mathbb{Q}}
\newcommand{\C}{\mathbb{C}}
\newcommand{\F}{\mathbb{F}}
\newcommand{\X}{\mathcal{X}}
\newcommand{\D}{\mathcal{D}}
\newcommand{\Cont}{\mathcal{C}}

\renewcommand{\labelenumi}{(\alph{enumi})}
\newtheorem{maintheorem}{Theorem}[]
\renewcommand*{\themaintheorem}{\Alph{maintheorem}}
\newtheorem*{theorem*}{Theorem}
\newtheorem*{corollary*}{Corollary}
\newtheorem*{remark*}{Remark}
\newtheorem*{example*}{Example}
\newtheorem*{question*}{Question}
\newtheorem*{definition*}{Definition}
\newtheorem{conjecture}[maintheorem]{Conjecture}
\newtheorem*{conjecture*}{Conjecture}

\begin{abstract}
We prove the existence of three non-round, non-isometric Einstein metrics with positive scalar curvature on the sphere $S^{10}.$ Previously, the only even-dimensional spheres known to admit non-round Einstein metrics were $S^6$ and $S^8.$ 
\end{abstract}

\maketitle

\section*{Introduction}

A Riemannian manifold $(M,g)$ is called an Einstein manifold if its Ricci tensor satisfies $\Ric (g) = \lambda g$ for some $\lambda \in \R.$ 

The first non-round Einstein metrics on spheres were discovered in the 1970s by Jensen \cite{JensenEinsteinMetricsOnPrincipalFibreBundles} and Bourguignon-Karcher \cite{BourguignonKarcherCurvOperatorsPinchingEstimates} on $S^{4m+3},$ for $m\ge 1$, and $S^{15}$, respectively. These metrics carry transitive group actions by $Sp(m+1)$ resp. $\operatorname{Spin}(9)$ and, in fact, were shown in 1982 by Ziller \cite{ZillerHomogeneousEinsteinMetrics} to be all non-round homogeneous Einstein metrics on spheres. In particular, all homogeneous Einstein metrics on spheres of even dimension are round.

This led to the question of whether \emph{all} non-round Einstein metrics on spheres occur in odd dimensions, which was settled in the negative by B\"ohm in \cite{BohmInhomEinstein}, who constructed infinite (discrete) families of Einstein metrics on all spheres of dimension $n \in \{5, \ldots, 9 \}$, in particular on $S^6$ and $S^8$.

The only new Einstein metrics constructed on even-dimensional spheres since then are a nearly K\"ahler metric on $S^6$ due to Foscolo-Haskins \cite{FoscoloHaskinsNearlyKaehler} and a metric on $S^8$ due to Chi \cite{ChiPositiveEinsteinMetrics}. 
In particular, up until now, $S^6$ and $S^8$ are still the only even-dimensional spheres known to admit non-round Einstein metrics.

The main result of this paper is the addition of $S^{10}$ to the list of even-dimensional spheres admitting non-round Einstein metrics:

\begin{maintheorem}
\label{EinsteinMetricsOnSpheresMainTheorem}
The ten-dimensional sphere $S^{10}$ admits three non-round, non-isometric Einstein metrics of positive scalar curvature. 
\end{maintheorem}

In stark contrast to the even-dimensional case, the existence question has been resolved in all odd dimensions due to recent constructions using the framework of Sasaki-Einstein metrics.

Specifically, in \cite{BoyerGalickiKollarEinsteinMetricsOnSpheres}, Boyer-Galicki-Koll\'{a}r described Sasaki-Einstein metrics on $S^{4m+1}$ for all $m \geq 1$ as well as on $S^7,$ many even admitting continuous Sasaki-Einstein deformations. Sasaki-Einstein metrics on $S^{2m+1}$ for all $m \geq 2$ were constructed by Ghigi-Koll\'{a}r in \cite{GhigiKollarKaehlerEinsteinOrbifolds}. Collins-Sz\'{e}kelyhidi \cite{CollinsSzekelyhidiSasakiEinsteinMetricsKStability} later proved that $S^5$ admits infinitely many families of Sasaki-Einstein metrics and conjectured that the same is true for any odd dimensional sphere of dimension at least five. This was in fact established by Liu-Sano-Tasin in \cite{LiuSanoTasinIninitelySasakiEinsteinMetricsSpheres}. Noting that any Einstein metric in dimension three has constant sectional curvature, these results settle the existence question in all odd dimensions.

\vspace{2mm}

To prove Theorem \ref{EinsteinMetricsOnSpheresMainTheorem}, we look for cohomogeneity one Einstein metrics on $S^{10}$ that are invariant under the standard group actions by $SO(d_1+1)\times SO(d_2+1)$ for $(d_1, d_2)$ with $d_1+d_2=9.$ Noting that the case $(1,8)$ only supports the round metric, see Remark \ref{d1Remark}, we construct one non-round metric for each pair $(d_1, d_2)=$ $(2,7), (3,6)$ and $(4,5)$.

\vspace{2mm}

Based on numerical investigations, we conjecture that the metrics of Theorem \ref{EinsteinMetricsOnSpheresMainTheorem} are all cohomogeneity one Einstein metrics on $S^{10}$. This sets it apart from $S^6$ and $S^8$, where infinitely many such metrics exist. In fact, we conjecture

\begin{conjecture}
Only finitely many even-dimensional spheres admit non-round Einstein metrics of cohomogeneity one.
\end{conjecture}

The question of whether this should be expected even without assumptions on the cohomogeneity is unclear.

\vspace{2mm}

For comparison, the metrics of B\"ohm \cite{BohmInhomEinstein}, Foscolo-Haskins \cite{FoscoloHaskinsNearlyKaehler} and Chi \cite{ChiPositiveEinsteinMetrics} are also of cohomogeneity one, with actions by $SO(d_1+1)\times SO(d_2+1)$, with $5 \leq d_1 +d_2+1 \leq 9$ and $d_1, d_2 \geq 2,$ $SU(2)\times SU(2)$ and $Sp(2)Sp(1)$ respectively. A common feature in their construction is the reliance on symmetric solutions, i.e., solutions admitting a reflection symmetry through a principal orbit. Both B\"ohm \cite{BohmInhomEinstein} and Foscolo-Haskins \cite{FoscoloHaskinsNearlyKaehler} construct their Einstein metrics on spheres by first exhibiting symmetric Einstein metrics on the associated products $S^{d_1+1}\times S^{d_2}$ resp. $S^3\times S^3$. This relies on the counting principle developed by B\"ohm in \cite[Lemmas 4.4 and 4.5]{BohmInhomEinstein}, which makes it possible to find such symmetric solutions. They then combine the existence of these symmetric solutions with other techniques to deduce the existence of non-round Einstein metrics on their spheres. The Einstein metric on $S^8$ constructed by Chi in \cite{ChiPositiveEinsteinMetrics} is itself symmetric (with principal orbit $S^7$) and produced using the counting principle.

In contrast, numerical investigations suggest that in our setup there do not exist symmetric Einstein metrics on the associated products $S^{d_1+1}\times S^{d_2}$ with the single exception of $S^{3}\times S^{7}$, see Remark \ref{NumericalSolutions}. We therefore cannot rely on the counting principle to prove Theorem \ref{EinsteinMetricsOnSpheresMainTheorem}. As a replacement, we introduce a new technique based on a rotation index for curves that allows us to bypass this problem and also still lets us recover B\"ohm's metrics, see Remark \ref{boehmrecovery}.

\vspace{2mm}

We now provide details of the construction. The basic setup is as in B\"ohm's construction of Einstein metrics on $S^{d_1+d_2+1}$ for $d_1, d_2 \geq 2$ and $5 \leq d_1+d_2+1 \leq 9.$ In the following, we denote by $n=d_1+d_2$ the dimension of the principal orbit.

Due to the cohomogeneity one structure, away from the singular orbits $S^{d_2}$ and $S^{d_1}$, the metric is given by $dt^2 + f_1^2(t) g_{S^{d_1}} + f_2^2(t) g_{S^{d_2}}$ and the Einstein equation corresponds to a system of ordinary differential equations for $(f_1, \Dot{f}_1, f_2, \Dot{f}_2).$ The smooth collapse of the principal orbit to the singular orbit $S^{d_2}$ corresponds to the singular initial condition $(f_1, \Dot{f}_1, f_2, \Dot{f}_2)=(0,1,\Bar{f}_2,0)$ at $t=0,$ where the parameter $\Bar{f}_2>0$ controls the volume of the singular orbit. Such a trajectory induces  a smooth Einstein metric on $S^{d_1+d_2+1}$ if in addition there are $T>0$ and $\Bar{f}_1>0$ such that $(f_1, \Dot{f}_1, f_2, \Dot{f}_2)=(\Bar{f}_1,0,0,-1)$ at $t=T.$ Note that in this case the mean curvature of the principal orbit decreases from $+ \infty$ at $t=0$ to $- \infty$ at $t=T.$ 

The sine-suspension over the principal orbit $S^{d_1} \times S^{d_2}$ equipped with the Einstein metric $\frac{d_1-1}{n-1} g_{S^{d_1}} + \frac{d_2-1}{n-1} g_{S^{d_2}}$ is called {\em cone solution}. Specifically, it is given by $f_i(t)= \sqrt{\frac{d_i-1}{n-1}} \sin(t)$ for $t \in (0,\pi)$ and it is singular at both $t=0$ and $t =\pi.$ 

With suitable coordinate changes, we transform the Einstein equation into a regular differential equation on (a set homeomorphic to) the cylinder $\overline{D}_1(0) \times [-1,1] \subset \R^3.$ The interval here is parametrized by the rescaled mean curvature $H$, which decreases from $H=1$ to $H=-1$ along the Einstein ODE.
The smooth collapse of $S^{d_i}$ corresponds to a fixed point $p_i^{\pm} \in S^1 \times \{ H= \pm 1 \}.$ In particular, a trajectory in $D_1(0) \times (-1,1)$ connecting $p_1^+$ and $p_2^-$ corresponds to a smooth Einstein metric on $S^{d_1+d_2+1}$ with singular orbit $S^{d_2}$ at $H=1$ and $S^{d_1}$ at $H=-1.$ In these coordinates, the cone solution corresponds to the trajectory $(0,0,H)$ for $H \in (-1,1).$ We call $(0,0,\pm 1)$ the base points of the cone solution. These are fixed points of the ODE.
\begin{figure}[h]
    \centering
    \includegraphics[scale=0.42]{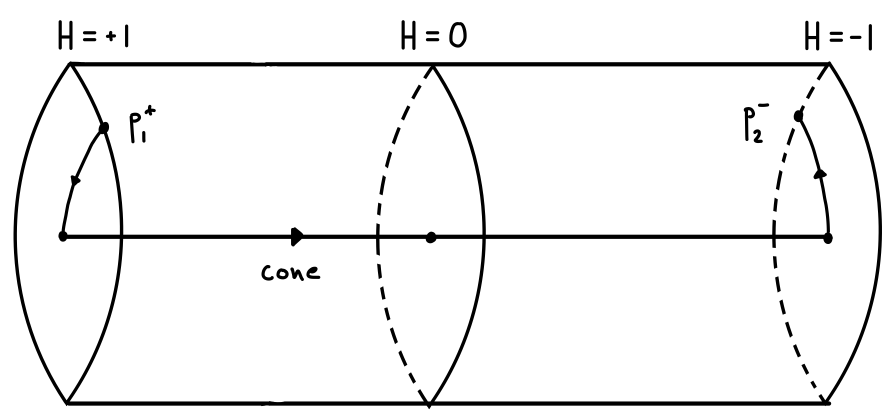}
    \caption{Coordinate Cylinder with selected trajectories}
    \label{fig:fig1}
\end{figure}

We note that the boundary parts $D_1(0) \times \{ \pm 1 \}$, $S^1 \times \{ \pm 1 \}$ and $S^1 \times (-1,1)$ of the cylinder are preserved by the ODE. Furthermore, all fixed points are hyperbolic and we denote by $M_i^+$ (resp. $M_i^-$) the part of the unstable manifold of $p_i^+$ (resp. the stable manifold of $p_i^-$), i.e., the union of all trajectories emanating from $p_i^+$ (resp. converging to $p_i^-$), in $\overline{D}_1(0) \times [-1,1].$ The $M_i^{\pm}$ are $2$-dimensional contractible surfaces. Since $(0,0,H)$ corresponds to the cone solution, we may introduce polar coordinates on $D_1 \times \{H = h \}$. This way we may pick an angle function $\varphi$ on $M_1^+$ such that $p_1^+ \in S^1 \times \{H = 1 \}$ has an angle $\varphi(p_1^+) \in (0, \frac{\pi}{2}).$

The trajectory $\gamma_i^{\text{RF}}$ in $M_i^+ \cap \{ H=1 \}$ connects $p_i^+$ and the origin in $D_1(0),$ i.e., the base point of the cone solution. This trajectory corresponds to the Ricci flat metric on $\R^{d_i+1} \times S^{d_j}$ discovered by B\"ohm \cite{BohmNonCompactComhomOneEinstein}. Furthermore, trajectories in $D_1(0) \times (-1,1)$ that come close to the base point of the cone solution near $H=1$ remain close to the cone solution as long as $H \geq 0.$ This is a special case of B\"ohm's Convergence Theorem \cite[Theorem 5.7]{BohmInhomEinstein}.

Note that Einstein metrics on $S^{d_1+d_2+1}$ correspond to intersection points of $M_1^+$ and $M_2^-$ in $\{ H = 0 \}.$ In particular, we can deduce the existence of Einstein metrics from the geometry of $M_i^{\pm} \cap \{H = 0\}.$ The key idea is the observation that one obtains two intersection points of $M_1^+$ and $M_2^-$ in $\{ H = 0 \}$ if the angle function $\varphi$ of $M_1^+$ attains values $\varphi > \frac{3}{2}\pi$ for points in the slice $M_i^{\pm} \cap \{ H = 0 \}$ near the cone solution, cf. Figure \ref{fig:fig2}.

To see this, we observe the following: The fixed point $p_1^+$ is in the first quadrant of $S^1 \times \{ H=1 \}$ and the trajectory which emanates from $p_1^+$ and lies in the boundary part $S^1 \times (-1,1)$ starts and remains in the first quadrant of the $S^1$-factor. Thus the curve $M_1^+ \cap \{ H = 0 \}$ meets the boundary of $\{ H = 0 \}$ in the first quadrant.  At the other end it approaches the cone solution, i.e., the origin in $D_1 \times \{ H= 0 \},$ as a consequence of B\"ohm's convergence theorem. By reversing the Einstein ODE, an analogous argument shows that the curve $M_2^- \cap \{ H = 0 \}$ has a similar geometry.
If the angle function on $M_1^+$ increases to $\varphi > \frac{3}{2} \pi$ at points in the slice $M_1^+ \cap \{ H = 0 \}$ near the cone solution, then both curves exhibit sufficient rotation around the cone solution to deduce the existence of at least two intersection points of the curves $M_1^+ \cap \{ H = 0 \}$ and $M_2^- \cap \{ H = 0 \}$. One of the intersection points then corresponds to the round metric, the other to a non-round Einstein metric.

Therefore it suffices to prove that the angle function $\varphi$ of $M_1^+$ indeed attains values $\varphi > \frac{3}{2}\pi$ in the slice $\{ H = 0\}$ at points near the cone solution. In dimension $n+1=10,$ the linearization at the base point of the cone solution shows that the trajectory $\gamma_1^{\text{RF}},$ and thus the curve $M_1^+ \cap \{ H =1 \},$ approaches the base point of the cone solution in a specific tangent direction. 

To illustrate the idea behind the proof, suppose that $M_i^+$ extends $C^2$-regularly to its boundary at the cone solution. In particular, each slice $M_i^+ \cap \{H = h \}$ then has a well-defined tangent direction $X_h$ at the cone solution. $X_h$ now satisfies the linearized Einstein ODE along the cone solution and $\lim_{h \to 1} X_h$ is the tangent direction of $\gamma_1^{\text{RF}}$ at the cone solution. The key idea of the proof is that thus the geometry of $M_1^+ \cap \{ H = h \}$ near the cone solution is determined by the tangent direction of $\gamma_1^{\text{RF}}$ and the behaviour of the linearized ODE along the cone solution. In particular, one obtains $\varphi > \frac{3}{2} \pi$ near the cone solution in $M_1^+ \cap \{ H = 0 \}$ provided $X_h$ rotates sufficiently around the cone solution as $h$ decreases from $h=1$ to $h=0.$ When writing $X_h$ in polar coordinates, its angle satisfies the ODE \eqref{AngleODEeq1}. The required estimate on the angle then follows from a direct ODE comparison argument as in Proposition \ref{AngleBarrier}. 

\begin{figure}[!htb]
     \centering
     \includegraphics[scale = .4]{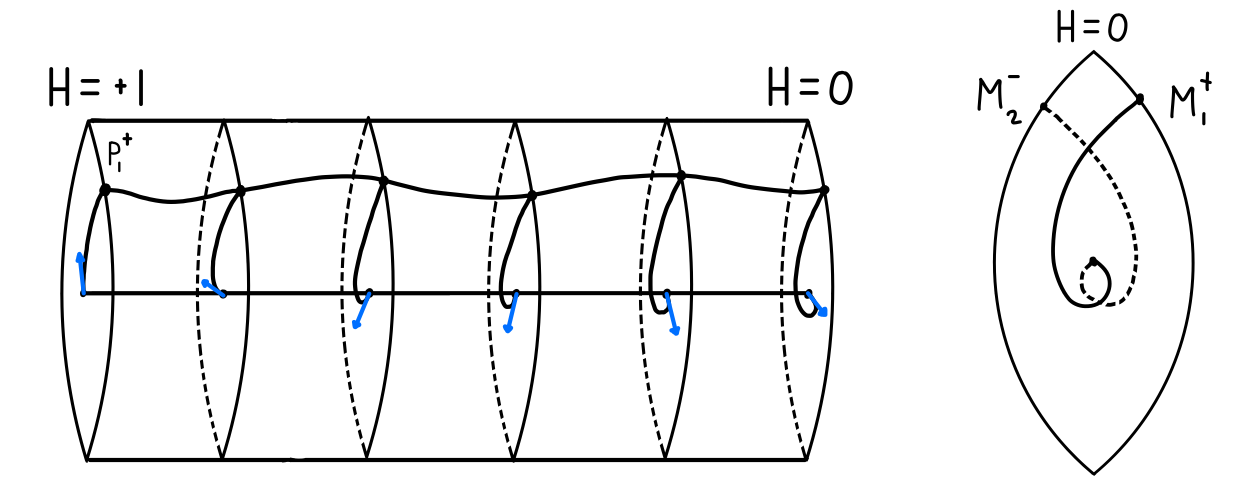}
     \caption{Left: Twisting of the curve $M_1^+\cap \{H=h\}$, with tangent direction vector field $X$ in blue. \\Right: The slice $H=0$. $M_2^-$ twists opposite to $M_1^+$. If there is enough twisting, the two curves must intersect at least twice.}\label{fig:fig2}
\end{figure}

With regard to the technical execution of the argument, we make the following two remarks. First we note that the regularity assumption on $M_i^+$ at the cone solution is difficult to verify and we rely instead on quantitative estimates. Thus, to estimate the winding angle of $M_1^+ \cap \{ H = h\}$ around the cone solution, we define a helicoidal surface by solving an adjusted angle ODE with $\varphi_0 > \frac{3}{2} \pi$ in the slice $\{ H=0 \}.$ Then we show that the surface passes the tangent direction of $M_1^+ \cap \{ H = 1 \}$ at the origin. The slight adjustment is necessary to ensure that the trajectories of the nonlinear Einstein ODE remain on one side of the helicoidal surface and thus obey the same rotational behaviour. This suffices to deduce the required geometry of the curve $M_1^+ \cap \{ H = 0 \}.$

To then deduce the existence of a given number of intersections points, we need to control in addition the behaviour of the intersection point of $M_1^+\cap \{H=h\}$ with the boundary, as this could in principle cancel out the rotational behaviour around the cone solution. For this it is convenient to have coordinates that exist far away from the cone solution. Finally we use a winding number argument on an associated closed curve to produce intersections. \vspace{2mm}

\textit{Structure.} In the preliminary section \ref{SectionPreliminaries} we recall the basic setup of cohomogeneity one Einstein manifolds and in particular the case of $SO(d_1+1) \times SO(d_2+1)$-invariant Einstein metrics on $S^{d_1+d_2+1}.$ In section \ref{SectionCoordinateChanges} we carry out the coordinate change to the cylinder. We define $M_i^{\pm}$, recover B\"ohm's Convergence Theorem and compute the tangent direction of $\gamma_i^{\text{RF}}$ at the base point of the cone solution. In section \ref{SectionRotationIndex} we define a generalized rotation index and we prove the key Lemma \ref{ThetaBoundsIntersections} on the number of intersection points. In section \ref{SectionODEestimates} we prove key inequalities for barrier solutions of the angle ODE \eqref{AngleODEeq1} and the adjusted angle ODE \eqref{DeltaAngleODE}, respectively. Section \ref{SectionProofOfMainTheorem} contains the proof of Theorem \ref{EinsteinMetricsOnSpheresMainTheorem} as well as concluding remarks. In particular, we comment on numerical solutions and indicate how one obtains an independent construction of B\"ohm's Einstein metrics on $S^5, \ldots, S^9$ using our methods. \vspace{2mm}

\textit{Acknowledgments.} We would like to thank Christoph B\"ohm for pointing out to us that he numerically observed the existence of Einstein metrics on $S^{10}$ during his work on \cite{BohmInhomEinstein}. We also thank Christoph B\"ohm, Lukas Fu{\ss}angel, Peter Petersen and Marco Radeschi for constructive comments on an earlier version of this paper. Furthermore we would like to thank the referees for numerous helpful comments.

JN acknowledges support by the Alexander von Humboldt Foundation through Gustav Holzegel's Alexander von Humboldt Professorship endowed by the Federal Ministry of Education and Research. Both authors were funded by the Deutsche Forschungsgemeinschaft (DFG, German Research Foundation) under Germany's Excellence Strategy EXC 2044–390685587, Mathematics M\"unster: Dynamics–Geometry–Structure.

\section{Preliminaries}
\label{SectionPreliminaries}

\subsection{Cohomogeneity one Einstein manifolds}

Let $(M,g)$ be an $(n+1)$-dimensional Riemannian manifold. $(M,g)$ is called Einstein if
\begin{equation*}
    \Ric = \lambda g
\end{equation*}
for some $\lambda \in \R.$

Suppose that a Lie group $G$ acts isometrically on $(M,g)$ with cohomogeneity one, i.e., the orbit space $M/G$ is one-dimensional. Any Einstein manifold of cohomogeneity one is either flat or has positive scalar curvature, \cite{BerardSurNouvellesVariteEinstein}. If $\lambda>0,$ due to Myers' theorem, $M$ is compact with finite fundamental group and thus $M/G$ is necessarily a compact interval. 

Away from the singular orbits, we may parametrize the metric as $dt^2 + g_t$, where $g_t$ is a family of $G$-invariant metrics on the principal orbit. Let $L_t$ denote the shape operator and $r_t$ the Ricci curvature of the principal orbit. Then,
by the work of Eschenburg-Wang \cite{EschenburgWangInitialValueProblem}, 
\begin{align}
\label{EvolutionMetric}
    \Dot{g}_t = 2 g_t L_t
\end{align}
and the Einstein equations are given by
\begin{align}
    \label{DerivativeL}
    - \Dot{L} -  \tr( L ) L + r & = \lambda \id, \\
    \label{DerivativeTrL}
    -\tr( \Dot{L} ) - \tr( L^2 ) & = \lambda, \\
    \Ric(X,N) & = 0, \nonumber
\end{align}
where $X$ is tangent to the principal orbit and $N$ is a horizontal lift of $\frac{\partial}{\partial t},$ a unit speed vector field on $M/G.$ Furthermore, any solution to \eqref{DerivativeL}, \eqref{DerivativeTrL} satisfies the constraint equation
\begin{align}
\label{GeneralConstraintEquation}
    \tr ( L^2 ) + \tr (r) - \tr(L)^2 = (n-1) \lambda.
\end{align}

Conversely, any solution to \eqref{EvolutionMetric}, \eqref{DerivativeL} that extends $C^{3}$-regularly to a singular orbit of dimension strictly less than that of the principal orbit induces a (possibly incomplete) Einstein metric on the associated disk bundle, see \cite[Corollary 2.6]{EschenburgWangInitialValueProblem}.

\subsection{Einstein metrics on spheres} 
\label{SectionEinsteinMetricsOnSpheresIntro}

From now on we consider $SO(d_1+1) \times SO(d_2+1)$-invariant Einstein metrics on the spheres $S^{d_1+d_2+1}$, for $d_1, d_2 \geq 2$.
This setup was considered by B\"ohm in \cite{BohmInhomEinstein} for $n=d_1+d_2 \leq 8.$ The metric is given by $dt^2 + g_t$, where
\begin{align*}
    g_t = f_1^2(t) \ g_{S^{d_1}} + f_2^2(t) \ g_{S^{d_2}}
\end{align*}
is the metric on the principal orbit $S^{d_1} \times S^{d_2}$ and $g_{S^{d_i}}$ denotes the round metric on $S^{d_i}.$ The shape operator and the Ricci curvature of the principal orbit satisfy 
\begin{align*}
    L = \left( \frac{\Dot{f}_1}{f_1} \id_{d_1}, \frac{\Dot{f}_2}{f_2} \id_{d_2} \right), \ \ 
    r= \left( \frac{d_1 -1}{f_1^2} \id_{d_1}, \frac{d_2 -1}{f_2^2} \id_{d_2} \right).
\end{align*}
In particular, \eqref{EvolutionMetric} is always satisfied. B\"ohm proved in \cite[Theorem 2.3]{BohmInhomEinstein} that for any $\Bar{f}_2>0$ there is a unique solution $(f_1, \Dot{f}_1, f_2, \Dot{f}_2)$ to \eqref{DerivativeL} with initial condition
\begin{align}
\label{InitialSmoothCollapseSDOneFcoords}
    f_1(0)=0, \ \Dot{f}_1(0)=1, \ f_2(0)=\Bar{f}_2, \ \Dot{f}_2(0)=0.
\end{align}
The initial condition corresponds to a smooth collapse of the $S^{d_1}$ and the solution induces a smooth Einstein metric on the disk bundle $D^{d_1 + 1} \times S^{d_2}.$ Solutions that in addition satisfy
\begin{align}
\label{TerminalSmoothCollapseSDTwoFcoords}
    f_1(T)=\Bar{f}_1, \ \Dot{f}_1(T)=0, \ f_2(T)=0, \ \Dot{f}_2(T)=-1
\end{align}
for some $T>0$ and $\Bar{f}_1>0$ induce smooth Einstein metrics on $S^{d_1+d_2+1}.$ \vspace{2mm}

\begin{definition}
\label{ConeSolution}
    \normalfont
    An important singular solution to the Einstein equations is the {\em cone solution}  given by
\begin{align*}
    dt^2 + \frac{d_1-1}{n-1} \sin(t)^2 g_{S^{d_1}} + \frac{d_2-1}{n-1} \sin(t)^2  g_{S^{d_2}}
\end{align*}
for $t \in (0,\pi).$    
\end{definition}

\section{Coordinate changes}
\label{SectionCoordinateChanges}

Fix $\lambda >0$. Following Chi \cite{ChiPositiveEinsteinMetrics}, set 
\begin{align*}
    \mathcal{L}  = \frac{1}{\sqrt{ \tr(L)^2+n \lambda}}, \ \
    \frac{d}{ds}  = \mathcal{L} \ \frac{d}{dt}.
\end{align*}
The derivative with respect to $s$ will be denoted by prime $'.$ 

\subsection{$(X,Y,H)$-coordinates} 
\label{SectionHXYCoordinates}
We introduce the coordinates
\begin{align*}
    X_i  = \mathcal{L} \ \frac{\Dot{f}_i}{f_i}, \ \
    Y_i  =  \frac{\mathcal{L} }{f_i}, \ \
    H  = \mathcal{L} \ \tr(L).
\end{align*}
Note that 
\begin{align*}
    H = \sum_{i=1}^2 d_i & X_i, \ \ 
    n \lambda \mathcal{L}^2 = 1 - \frac{\tr(L)^2}{\tr(L)^2 +n \lambda} = 1 - H^2 \ \text{ and } \\
    & \mathcal{L}^{'} = \mathcal{L} H \left( \sum_{i=1}^2 d_i X_i^2 + \frac{1}{n}(1-H^2) \right).
\end{align*}
It follows from \eqref{DerivativeL} and \eqref{DerivativeTrL} that 
\begin{align}    
    \label{EinsteinODEinHXYequ}
    X_j^{'} & = X_j H \left( \sum_{i=1}^2 d_i X_i^2 + \frac{1}{n} ( 1-H^2 ) - 1 \right) + (d_j-1) Y_j^2 - \frac{1}{n} ( 1 - H^2 ), \\
    \nonumber
    Y_j^{'} & = Y_j \left( H \left( \sum_{i=1}^2 d_i X_i^2 + \frac{1}{n} ( 1-H^2 ) \right) - X_j \right), \\
    \nonumber
    H^{' } & = (H^2-1) \left( \sum_{i=1}^2 d_i X_i^2 + \frac{1}{n} ( 1-H^2 ) \right)
\end{align}
for $j=1,2$. 

Similarly, a straightforward computation yields that
\begin{align*}
    \mathcal{S}_1 & = \sum_{i=1}^2 d_i X_i^2 + \sum_{i=1}^2d_i (d_i-1) Y_i^2 - \frac{1}{n} H^2 - \frac{n-1}{n}, \\
    \mathcal{S}_2 & =  \sum_{i=1}^2 d_i X_i - H
\end{align*}
satisfy
\begin{align*}
   \frac{1}{2} \mathcal{S}_1^{'} & = \mathcal{S}_1 H \left( \sum_{i=1}^2 d_i X_i^2 + \frac{1}{n} (1 - H^2) \right) - \frac{1}{n} (1 - H^2) \mathcal{S}_2, \\
    \mathcal{S}_2^{'} & = \mathcal{S}_1 + \mathcal{S}_2 H \left( \sum_{i=1}^2 d_i X_i^2 + \frac{1}{n} (1 - H^2) -1 \right).
\end{align*}

Note that any solution $(X_j, Y_j,H)$ to \eqref{EinsteinODEinHXYequ} that is induced by an Einstein metric satisfies $\mathcal{S}_1=\mathcal{S}_2 = 0.$ In particular, the constraint equation \eqref{GeneralConstraintEquation} corresponds to $\mathcal{S}_1 \equiv 0.$

\begin{proposition}
\label{RecoverMetricFromHXY}
A solution $(X_j, Y_j,H)$ to \eqref{EinsteinODEinHXYequ} with $\mathcal{S}_1=\mathcal{S}_2 = 0$, $H \in (-1,1)$ and $Y_1, Y_2 >0$ at some $s_0 \in \R$ induces a solution to \eqref{EvolutionMetric} - \eqref{DerivativeTrL} by setting
\begin{align*}
    \mathcal{L} = \sqrt{\frac{1 - H^2}{n \lambda}}, \ t=t(s_0) + \int_{s_0}^s \mathcal{L} ( \tau ) d \tau, \ \frac{\Dot{f}_i}{f_i} = \frac{X_i}{\mathcal{L}}, \ \text{and} \ f_i=\frac{\mathcal{L}}{Y_i}
\end{align*}
for $i=1,2.$
\end{proposition}
\begin{proof}
    Note that the conditions $\mathcal{S}_1=\mathcal{S}_2 = 0$, $H \in (-1,1)$ and $Y_1, Y_2 >0$ are preserved. In particular, $\mathcal{L}$ is well defined and satisfies the corresponding ODE. This implies \eqref{EvolutionMetric}. Since $H=\sum_{i=1}^2 d_i X_i,$ one easily recovers \eqref{DerivativeL} and \eqref{DerivativeTrL}.
\end{proof}

\subsection{$(Y, \Delta,H)$-coordinates} 
\label{SubsectionHYDeltaCoordinates}
Let 
\begin{align*}
    \Delta = X_1 - X_2.
\end{align*}
Any given $(\Delta, H)$ uniquely determines $(X_1, X_2)$ such that $\mathcal{S}_2=0$ via $X_1 = \frac{1}{n} ( d_2 \Delta + H),$ $X_2 = \frac{1}{n} ( -d_1 \Delta + H).$ This transforms between $(X,Y,H)$-coordinates and $(Y, \Delta, H)$-coordinates.
With the observation that 
\begin{align*}
    \sum_{i=1}^2 d_i X_i^2 = \frac{1}{n} ( d_1 d_2 \Delta^2 + H^2 ),
\end{align*}
it is straightforward to compute that \eqref{EinsteinODEinHXYequ} is equivalent to 
\begin{align}
    \label{EinsteinODEinHYDelta}
    Y_1^{'} & = \frac{d_1 d_2 }{n} Y_1 \Delta \left( \Delta H - \frac{1}{d_1} \right), \\
    \nonumber
    Y_2^{'} & = \frac{d_1 d_2 }{n} Y_2 \Delta \left( \Delta H + \frac{1}{d_2} \right), \\
    \nonumber
    \Delta^{'} & = \frac{\Delta H}{n} \left( d_1 d_2 \Delta^2 - (n-1) \right) + (d_1-1) Y_1^2 - (d_2 -1) Y_2^2, \\
    \nonumber
    H^{'} & = - \frac{1-H^2}{n} \left( d_1 d_2 \Delta^2 +1 \right).
\end{align}
By Proposition \ref{RecoverMetricFromHXY}, Einstein metrics correspond to solutions of \eqref{EinsteinODEinHYDelta} satisfying in addition
\begin{align*}
    H^2 < 1,\ \
    Y_i >0, \ \
    \mathcal{S}_1 = \sum_{i=1}^2 d_i (d_i-1) Y_i^2 + \frac{d_1 d_2}{n} \Delta^2 - \frac{n-1}{n} = 0.
\end{align*}

\begin{proposition}
     \label{EinsteinLocus}
    The set
    \begin{align*}
        \mathcal{S} = \left\lbrace (Y_1,Y_2,\Delta, H)\in \R^4 \ | \ Y_i\ge 0, \ H^2\le 1 \right\rbrace \cap \left\lbrace \mathcal{S}_1 = 0 \right\rbrace
    \end{align*}
    is invariant under the Einstein ODE \eqref{EinsteinODEinHYDelta}, as is its boundary.
    Moreover, solutions in $\mathcal{S}$ exist for all times. The coordinate $H$ is non-increasing on $[-1,1]$ and in fact decreasing for $H \ne \pm 1$.
\end{proposition}
\begin{proof}
    The invariance of $\mathcal{S}$ and its boundary is immediate. Note that $\mathcal{S}$ is compact since the highest order term in $\mathcal{S}_1$ is a positive definite quadratic form in $(Y, \Delta)$. This implies long time existence of solutions in $\mathcal{S}.$
\end{proof}

\begin{proposition}
\label{VolumeMonotonic}
    The function $Y_1^{2 d_1} Y_2^{2 d_2}$ is non-decreasing along \eqref{EinsteinODEinHYDelta} for $H \geq 0.$
\end{proposition}
\begin{proof}
    This is immediate from $\left( Y_1^{2 d_1} Y_2^{2 d_2} \right)^{'} = 2 d_1 d_2 Y_1^{2 d_1} Y_2^{2 d_2} \Delta^2 H \geq 0.$
\end{proof}
Geometrically, $Y_1^{d_1} Y_2^{d_2}$ corresponds to the volume of the principal orbit. In particular, at $H=0$ we have the unique maximal volume orbit. 

\subsection{$(Z, \Delta,H)$-coordinates} 
\label{SubsectionYDeltaHcoordinates} 

Let 
\begin{align*}
    Z = (d_1-1) Y_1^2 - (d_2 -1) Y_2^2.
\end{align*}
For trajectories in $\mathcal{S}$ we can recover $Y_1, Y_2$ via the constraint equation $\mathcal{S}_1=0$ from
\begin{align*}
    d_1 Z + n (d_2-1) Y_2^2 + \frac{d_1 d_2}{n} \Delta^2 & = \frac{n-1}{n}, \\
    n(d_1-1) Y_1^2 - d_2 Z + \frac{d_1 d_2}{n} \Delta^2 & = \frac{n-1}{n}. 
\end{align*}
Furthermore, the constraint equation implies that 
\begin{align*}
    Z^{'} = & \ \frac{2 d_1 d_2}{n} \Delta \left( (d_1-1) Y_1^2 ( \Delta H - \frac{1}{d_1} ) - (d_2 - 1) Y_2^2 ( \Delta H + \frac{1}{d_2} ) \right) \\
    = & \ \frac{2 d_1 d_2}{n} Z \Delta^2 H - \frac{2}{n} \Delta \left( (d_1-1) d_2 Y_1^2 + d_1 (d_2 -1) Y_2^2 \right) \\
    = & \ \frac{2 d_1 d_2}{n} Z \Delta^2 H - \frac{2}{n} \Delta \left( d_1 (d_1-1) Y_1^2 + d_2 (d_2 -1) Y_2^2 + (d_2-d_1) Z \right) \\    
    = & \ \frac{2}{n} \Delta \left( d_1 d_2 Z \Delta H + \frac{d_1 d_2}{n} \Delta^2 - \frac{n-1}{n} + (d_1 -d_2) Z \right).    
\end{align*}

Thus, in $(Z, \Delta, H)$-coordinates, the set $\mathcal{S}$ is given by
\begin{align*}
   \mathcal{S} = \left\lbrace d_1 Z + \frac{d_1 d_2}{n} \Delta^2 \leq \frac{n-1}{n} \right\rbrace \cap \left\lbrace -d_2 Z + \frac{d_1 d_2}{n} \Delta^2 \leq \frac{n-1}{n} \right\rbrace \cap \left\lbrace H^2 \le 1 \right\rbrace
\end{align*}
and within $\mathcal{S}$ the Einstein ODE \eqref{EinsteinODEinHYDelta} is equivalent to 
\begin{align}
\label{EinsteinODEinZDeltaH}
    Z^{'} & = \frac{2}{n} \Delta \left( d_1 d_2 Z \Delta H + \frac{d_1 d_2}{n} \Delta^2 - \frac{n-1}{n} + (d_1 -d_2) Z \right), \\
    \nonumber
    \Delta^{'} & = \Delta H \left( \frac{d_1 d_2}{n} \Delta^2 - \frac{n-1}{n} \right) +Z, \\
    \nonumber
    H^{'} & = - \frac{1-H^2}{n} \left( d_1 d_2 \Delta^2 +1 \right).
\end{align}

\begin{remark}
\normalfont
    \label{BoundaryPreserved}
    Proposition \ref{EinsteinLocus} implies that $\mathcal{S}$ and its boundary are preserved. In fact, one can also check directly that both 
\begin{align*}
    d_1 Z + \frac{d_1 d_2}{n} \Delta^2 = \frac{n-1}{n} \ \text{and} \ -d_2 Z + \frac{d_1 d_2}{n} \Delta^2 = \frac{n-1}{n}
\end{align*}
are preserved. 
\end{remark}

\begin{proposition}
\label{FixedPointsInHDeltaZ}
Restricted to $\mathcal{S}$, the fixed points of the Einstein ODE \eqref{EinsteinODEinZDeltaH} are given in $(Z, \Delta, H)$-coordinates by
\begin{enumerate}
    \item $p_1^+= \left( \frac{d_1-1}{d_1^2}, \frac{1}{d_1},1 \right)$ and $p_1^-= \left( \frac{d_1 -1}{d_1^2}, - \frac{1}{d_1}, -1 \right).$ These correspond to a smooth collapse of $S^{d_1}.$
    \item $p_2^+= \left( - \frac{d_2-1}{d_2^2}, - \frac{1}{d_2},1 \right)$ and $p_2^-= \left( - \frac{d_2 -1}{d_2^2}, \frac{1}{d_2}, -1 \right).$ These correspond to a smooth collapse of $S^{d_2}.$
    \item $\cone^{\pm}= \left( 0,0, \pm 1 \right).$ These correspond to the base points of the cone solution. 
    \item $q_1^{\pm}= \left( 0, \mp \sqrt{\frac{n-1}{d_1 d_2}}, \pm 1 \right), q_2^{\pm}= \left( 0, \pm \sqrt{\frac{n-1}{d_1 d_2}}, \pm 1 \right).$ These correspond to singular solutions.   
\end{enumerate}
\end{proposition}
\begin{proof}
Clearly, any fixed point satisfies $H^2=1.$ If $\Delta = 0,$ then $Z=0.$

Otherwise $\Delta H \neq 0$ and $\Delta^{'}=0$ implies $\frac{d_1 d_2}{n} \Delta^2 - \frac{n-1}{n} = - \frac{Z}{\Delta H}.$ $Z^{'}= 0$ thus yields $Z \left( d_1 d_2 \Delta H - \frac{1}{\Delta H} + d_1 - d_2 \right) = 0.$

If $Z=0,$ then $\Delta^2 = \frac{n-1}{d_1 d_2}.$ Otherwise $\Delta H = \frac{1}{d_1}$ or $-\frac{1}{d_2}$ and one computes the value of $Z$ from $Z = \Delta H \ \left( \frac{n-1}{n} - \frac{d_1 d_2}{n} ( \Delta H )^2 \right).$
 
 With regard to the smoothness conditions, note for example that the smoothness conditions \eqref{InitialSmoothCollapseSDOneFcoords}, \eqref{TerminalSmoothCollapseSDTwoFcoords} correspond to the fixed points $p_1^+,$ $p_2^-.$ 

One obtains the base points of the cone solution by converting the cone solution of definition \ref{ConeSolution} into $(Z, \Delta, H)$-coordinates and by taking the limits $s \to \pm \infty$ (which corresponds to $t \searrow 0,$ $t \nearrow \pi,$ respectively).

Solutions emanating from or converging to a fixed point in (d) correspond to incomplete metrics. 
\end{proof}

Note that all fixed points lie on the boundary of $\mathcal{S}.$

\begin{proposition}\label{LinearizedBehaviour}
    All fixed points of the Einstein ODE \eqref{EinsteinODEinZDeltaH} within $\mathcal{S}$ are hyperbolic. Furthermore:
    \begin{enumerate}
        \item $q_i^+$ are sources.
        \item $q_i^-$ are sinks.
        \item $p_i^{\pm}$ are saddles.
        \item The unstable manifolds of $p_i^+$ are $2$-dimensional and intersect $\partial \mathcal{S}$ transversally near $p_i^+$.
        \item The stable manifolds of $p_i^-$ are $2$-dimensional and intersect $\partial \mathcal{S}$ transversally near $p_i^-$.
        \item $\cone^{\pm}$ are saddle points of the ODE \eqref{EinsteinODEinZDeltaH}. Furthermore,
        \begin{itemize}
            \item[$\bullet$] $\cone^+$ is a sink for the 2-dimensional ODE obtained by restricting \eqref{EinsteinODEinZDeltaH} to the invariant set $\mathcal{S}\cap \{H=1\}$.
            \item[$\bullet$] $\cone^-$ is a source for the 2-dimensional ODE obtained by restricting \eqref{EinsteinODEinZDeltaH} to the invariant set $\mathcal{S}\cap \{H=-1\}$.
        \end{itemize}
    \end{enumerate}
\end{proposition}
\begin{proof}
    This is obtained by computing the linearization of \eqref{EinsteinODEinZDeltaH}.

    For example, at $p_1^+$ the linearization has eigenvalues $\frac{2}{d_1}$ with multiplicity $2$ and $-\frac{d_1-1}{d_1}$ with multiplicity $1$. The tangent space of the unstable manifold is spanned by $(
        1+\frac{d_1-d_2}{d_1 n}, 1, 0 )$ and $(\frac{d_1-1}{d_1^2}, 0, 1)$ and the tangent space of the stable manifold is spanned by $(-\frac{2}{n}\frac{d_2}{d_1},1,0).$
\end{proof}

\begin{remark} \normalfont
    The dimensions in Proposition \ref{LinearizedBehaviour} (d)-(e) are in agreement with section \ref{SectionEinsteinMetricsOnSpheresIntro},
    with the one-parameter family of solutions corresponding to \eqref{InitialSmoothCollapseSDOneFcoords} foliating the part of the $2$-dimensional unstable manifold of $p_1^+$ in the interior of $\mathcal{S}$.
\end{remark}

Furthermore, we have the following explicit solutions:

\begin{proposition}
    \label{SpecialSolutions}
    In $(Z, \Delta, H)$-coordinates, some special solutions to the Einstein ODE \eqref{EinsteinODEinZDeltaH} are given by:
    \begin{enumerate}
        \item The \textit{cone solution} $\cone(h) := (0,0,h)$,
        
        where $h(s)= -\tanh(s/n)$.
        \item The \textit{singular solution} $q_1^+(h)=q_2^-(h) := (0, - \sqrt{\frac{n-1}{d_1 d_2}} , h)$, 
        
        where $h(s)= -\tanh(s)$.
        \item The \textit{singular solution} $q_2^+(h)=q_1^-(h) := (0, \sqrt{\frac{n-1}{d_1 d_2}}, h)$, 
        
        where $h(s)= -\tanh(s)$.
    \end{enumerate}
\end{proposition}

\begin{definition}
    Let $M_i^+$ (resp. $M_i^-$) denote the intersection of the unstable manifold of $p_i^+$ (resp. the stable manifold of $p_i^-$) with $\mathcal{S}$.
\end{definition}

For example, the one-parameter family of solutions with initial condition \eqref{InitialSmoothCollapseSDOneFcoords} parametrizes the trajectories in $M_1^+ \cap \operatorname{int}(\mathcal{S}).$

\begin{corollary}
\label{UniqueIntersectionWithBoundary}
$M^{+}_i$ (resp. $M_i^-$) intersects the boundary parts $\partial\mathcal{S}\cap \{H^2 < 1\}$ and $\partial\mathcal{S}\cap \{H =  1\}$ (resp. $\partial\mathcal{S}\cap \{H =  -1\}$) in a single trajectory each.
\end{corollary}
 \begin{proof}
    Since the boundary is invariant, any trajectory in $M_i^+$ that intersects a boundary part is in fact contained in that boundary part. The linearization at $p_i^+$ implies that there is a unique direction from $p_i^+$ into that boundary part.
\end{proof} 

In fact, Proposition \ref{QuadrantRotation} implies that the trajectory in $M_i^+ \cap \partial\mathcal{S}\cap \{H^2 < 1\}$ connects $p_i^+$ and $q_i^-.$ 

Before we continue, we define a helpful auxiliary transformation on $\mathcal{S}$:

\begin{definition}\label{defBarInvolution}
    In $(Z, \Delta, H)$-coordinates, consider the involution given by $\overline{(Z, \Delta, H)} =  (Z, -\Delta, H)$. 
    
    For a set $X$, we denote the image of $X$ under this map by $\overline{X}$.
\end{definition}

\begin{lemma}
\label{IntersectionsAreSolutions}
Let $d_1, d_2\ge 2$ be natural numbers. Modulo scaling and with respect to the standard action, $SO(d_1+1)\times SO(d_2+1)$-invariant Einstein metrics on $S^{d_1 + d_2 +1}$ are in one-to-one correspondence with points in $M_1^+ \cap \overline{M_2^+} \cap \{H=0\}$.
\end{lemma}

\begin{proof}
Let $g$ be an $SO(d_1+1)\times SO(d_2+1)$-invariant Einstein metric on $S^{d_1+d_2+1}$. Recall that then the Einstein $\lambda$ constant is necessarily positive, \cite{BerardSurNouvellesVariteEinstein}. 
    
As explained in section \ref{SectionPreliminaries}, we may represent the metric as $dt^2 + f_1^2 g_{S^{d_1}} + f_2^2 g_{S^{d_2}}$ for two non-negative functions $f_1, f_2$ on $S^{d_1+d_2+1}$ with appropriate smoothness conditions at the singular orbits and $f_1, f_2 > 0$ away from the singular orbits. By choosing a $t$-direction from one singular orbit to the other, the Einstein condition becomes an ODE for $(f_1, \Dot{f}_1, f_2, \Dot{f}_2)$. Since $\lambda >0$, we may convert the Einstein equations for $(f_1, \Dot{f}_1, f_2, \Dot{f}_2)$ into the Einstein ODE \eqref{EinsteinODEinZDeltaH} in $(Z, \Delta, H)$-coordinates.
    
Note that the choice of the $t$-direction corresponds to the symmetry of the Einstein ODE \eqref{EinsteinODEinZDeltaH} given in $(Z, \Delta, H)$-coordinates by $\sigma:(Z, \Delta, H) \mapsto (Z, -\Delta, -H)$. Furthermore, smoothness at a singular orbit corresponds to a trajectory emanating from or converging to one of the fixed points $p_i^{\pm}$ of Proposition \ref{FixedPointsInHDeltaZ}. 

Since the manifold is assumed to be a sphere, the principal orbit $S^{d_1} \times S^{d_2}$ must collapse to different factors at the singular orbits, i.e., both $S^{d_1}$ and $S^{d_2}$ occur as singular orbits. Thus we obtain either a trajectory from $p_1^+$ to $p_2^-$ or from $p_2^+$ to $p_1^-$. We choose the $t$-direction in the way that places us in the first case. 

This gives us a one-to-one correspondence between $SO(d_1+1)\times SO(d_2+1)$-invariant Einstein metrics on $S^{d_1+d_2+1}$ with a fixed Einstein constant $\lambda>0$ and trajectories of the Einstein ODE \eqref{EinsteinODEinZDeltaH} from $p_1^+$ to $p_2^-$ that are contained in the interior of $\mathcal{S}$, which is the region of the $(Z,\Delta,H)$-coordinate system corresponding to $f_1, f_2 > 0$.
 
By monotonicity of $H$, any trajectory from $H=1$ to $H=-1$ is uniquely determined by its single intersection point with the slice $\{H=0\}$. Conversely, any point $p$ in $\{H=0\}$ uniquely determines a trajectory of the Einstein ODE. By definition of stable and unstable manifolds, this trajectory emanates from $p_1^+$ if and only if $p \in M_1^+$ and converges to $p_2^-$ if and only if $p\in M_2^-$. Furthermore, note that $\sigma$ interchanges $M_i^+$ and $M_i^-$ and that at $\{H=0\}$ we have $\sigma((Z, \Delta, 0)) = (Z, -\Delta, 0) = \overline{(Z, \Delta, 0)}$, so $M_2^- \cap \{H=0\} = \sigma(M_2^+) \cap \{H=0\} = \overline{M_2^+}\cap \{H=0\}.$ Finally, note that any point in $M_1^+ \cap \overline{M_2^+} \cap \{H=0\}$ must lie in the interior of $\mathcal{S}$ by Corollary \ref{UniqueIntersectionWithBoundary} and thus corresponds to a smooth metric. 
\end{proof}

In the remainder of the section we observe some important properties of the Einstein ODE, in particular with regard to convergence to and rotation around the cone solution. 

Note that any trajectory in $\mathcal{S}$ with $Z=\Delta=0$ satisfies $Z=\Delta=0$ for all times and thus is a reparametrization of the cone solution. Geometrically, Proposition \ref{QuadrantRotation} below says that solutions in $\mathcal{S}$ exhibit rotational behaviour around the cone solution up to quadrants.

\begin{proposition}
\label{QuadrantRotation}
Let $(Z, \Delta, H)$ be a solution to the Einstein ODE within $\mathcal{S}$ with $(Z,\Delta) \neq (0,0).$ If it enters a quadrant in the $(Z, \Delta)$-plane, it either remains there for all times or it exits the quadrant into the next quadrant going counterclockwise around the cone solution. 
\end{proposition}
\begin{proof}
If $\Delta = 0,$ then $\Delta^{'} = Z.$ In fact, we also have $Z^{'}=0$ so the $\{\Delta=0\}$-axis is a vertical nullcline. 

If $Z=0$, then $Z^{'}= \frac{2}{n} \Delta \left( \frac{d_1 d_2}{n} \Delta^2 - \frac{n-1}{n} \right).$ Note that $\frac{d_1 d_2}{n} \Delta^2 - \frac{n-1}{n} \leq 0$ within $\mathcal{S}$ with equality only along the singular solutions of Proposition \ref{SpecialSolutions}.
\end{proof}

\begin{remark}
\label{VolumeMonotonicZDeltaH}
\normalfont
The monotone quantity $Y_1^{2 d_1} Y_2^{2 d_2}$ of Proposition \ref{VolumeMonotonic} is given in $(Z,\Delta,H)$-coordinates (up to a multiplicative constant) by
    \begin{align*}
        \mathcal{F} =  \left( \frac{n-1}{n} - \frac{d_1 d_2}{n} \Delta^2 + d_2 Z \right)^{d_1}\left( \frac{n-1}{n} - \frac{d_1 d_2}{n} \Delta^2 - d_1 Z \right)^{d_2}.
    \end{align*}
Furthermore, 
    \begin{align*}
        \mathcal{F}^{'} = 2 d_1 d_2 \Delta^2 H \mathcal{F}
    \end{align*}
and $\mathcal{F}$ is non-decreasing for $H \geq 0.$ Note that $\mathcal{F}^{'} \equiv 0$ on the cone solution $\cone(h)$ and $\mathcal{F}$ has a local maximum on $\cone(h).$ We set $\mathcal{F}_{\cone} \colon = \mathcal{F}_{|\cone(h)}=\left( \frac{n-1}{n} \right)^n.$ 
\end{remark}

\begin{proposition}
\label{RicciFlatSystem}
    Every trajectory in the interior of $\mathcal{S} \cap \{ H = 1 \}$ converges to $\cone^+.$
\end{proposition}
\begin{proof}
    Since $\mathcal{S}$ is compact, the $\omega$-limit set of a trajectory is compact, connected, non-empty and invariant under the flow of the ODE. Since $\mathcal{F}$ is non-decreasing, $\mathcal{F}^{'}=0$ on the $\omega$-limit set. Thus, $\Delta=0$ and $\Delta^{'}=0$ then also implies $Z=0.$ In particular, the $\omega$-limit set consists of the fixed point $\cone^+.$
\end{proof}

The nontrivial trajectory in $M_1^+ \cap \{ H = 1 \}$ connects $p_1^+$ and $\cone^+$. It corresponds to the Ricci flat metric on $\R^{d_1+1} \times S^{d_2}$ originally constructed by B\"ohm in \cite{BohmNonCompactComhomOneEinstein}, cf. \cite{WinkSolitonsFromHopfFibrations}. Therefore, we denote the nontrivial trajectory in $M_i^+ \cap \{ H = 1 \}$ by $\gamma_{i}^{\text{RF}}.$ 

\begin{remark}
    \label{RemarkLinearizationAlongTheConeSolution}
    \normalfont
    In $(Z,\Delta,H)$-coordinates, the linearization along the cone solution is given by

\begin{align}
\label{LinearizationAlongConeSolution}
D(ODE)_{\cone(h)} =\begin{pmatrix}
0 & - \frac{2(n-1)}{n^2} & 0 \\
1 & - \frac{n-1}{n} h & 0 \\
0 & 0 & \frac{2}{n} h
    \end{pmatrix},
\end{align}
where $h(s)=-\tanh(s/n).$

Clearly, $\lambda_0=\frac{2}{n}h$ is an eigenvalue with eigenvector $(0,0,1)^T.$ The other eigenvalues are
\begin{align*}
    \lambda_{1,2} = \frac{\sqrt{n-1}}{2n} \left( - \sqrt{n-1} h \pm \sqrt{ (n-1) h^2 - 8 } \right)
\end{align*}
and the corresponding eigenvectors are $\left( - \frac{2(n-1)}{n^2}, \lambda_{1,2}, 0 \right)^T.$ 

Note that if $h=\pm \sqrt{\frac{8}{n-1}}$, then $\lambda_1 = \lambda_2= \mp \frac{\sqrt{2(n-1)}}{n}$ has algebraic multiplicity two but geometric multiplicity one, with corresponding eigenvector $(\sqrt{2(n-1)}, \pm n, 0)^T$.
\end{remark}

The Einstein ODE \eqref{EinsteinODEinZDeltaH} restricted to the invariant set $\mathcal{S} \cap \{ H = 1 \}$ is a $2$-dimensional ODE system in $(Z,\Delta).$ Thus, the linearization at $\cone^+=\cone(1)$ determines how trajectories in the interior of $\mathcal{S} \cap \{ H =1 \}$ approach $\cone^+,$ cf. \cite[Chapter 15]{CoddingtonLevinsonODEs}.

\begin{corollary}
\label{TangentDirectionAtCone+}
    Let $n=9$. Then, for each $i$, the trajectory $\gamma_i^{\text{RF}}(s)$ becomes tangent to the $(4,9)$-direction in the $(Z,\Delta)$-plane as it approaches $\cone^+.$
\end{corollary}

\begin{remark}\label{InvariantSets} 
\normalfont
Due to Remark \ref{VolumeMonotonicZDeltaH}, the sets 
\begin{align*}
    I_c = \{ (Z, \Delta, H) \in \mathcal{S} \ | \ \mathcal{F} \geq c, H \geq 0 \}
\end{align*}
are invariant under the Einstein ODE as long as $H \geq 0.$ Let $\varepsilon >0.$ Continuous dependence on the initial condition and Proposition \ref{RicciFlatSystem} imply that there are trajectories in $\mathcal{S} \cap \{ H^2 < 1 \}$ close to $\gamma_i^{\text{RF}}$ that enter $I_{\mathcal{F}_{\cone} - \varepsilon}$ and remain in $I_{\mathcal{F}_{\cone} - \varepsilon}$ as long as $H \geq 0.$ This is a special case of B\"ohm's Convergence Theorem \cite[Theorem 5.7]{BohmInhomEinstein}.
\end{remark}

\section{A generalized rotation index}
\label{SectionRotationIndex}

In this section we prove the key Lemma \ref{ThetaBoundsIntersections} on the winding number of curves in $\R^2$, which we will use in the proof of Theorem \ref{EinsteinMetricsOnSpheresMainTheorem} to produce intersections of $M_1^+ \cap \{ H=0 \}$ and $\conj{M_2^+}\cap \{H=0\}$.

\begin{definition}
\label{WindingAngle}
Let $\gamma:[a,b]\to \R^2\setminus \{0\}$ be a curve. We denote by $\theta(\gamma)$ the total winding angle that $\gamma$ makes around the origin. 

Explicitly, writing $\gamma(t)=|\gamma(t)|\begin{pmatrix}\cos\phi(t)\\ \sin \phi(t)\end{pmatrix}$, we set $\theta(\gamma)=\phi(b)-\phi(a)$.
\end{definition}

\begin{remark}
\label{WindingAngleMod2Pi}
\normalfont
    By the explicit characterization, it is clear that, for curves defined on compact intervals, $\theta(\gamma) \equiv \sphericalangle(\gamma(a), \gamma(b)) \text{ mod } 2\pi$. Furthermore, $\theta$ is continuous along homotopies in $\R^2\setminus \{0\}$.
\end{remark}

For our application, we will need to extend the definition to curves heading into the point we are counting winding around.

\begin{definition}\label{derivativecontrolswinding}
    Let $\gamma:[a,b)\to \R^2\setminus \{0\}$ be a curve that extends to a curve $\gamma: [a,b]\to \R^2$ with $\gamma(b)=0$. 
    
    We write $\theta(\gamma)\ge c$ if for all $\varepsilon>0$ small we have $\theta(\gamma_{|[a, b-\varepsilon]})\ge  c$. We write $\theta(\gamma) > c$ if $\theta(\gamma) \ge c'$ for some $c'>c$. 
    
    $\theta(\gamma)\le c$ and $\theta(\gamma) < c$ are defined analogously.
\end{definition}

We may now state the main Lemma of this section:

\begin{lemma}\label{ThetaBoundsIntersections}
Let $\gamma_{i}:[a_{i},b_{i})\to \R^2\setminus \{0\}$, $i=1,2,$ be two curves without self-intersection with $\gamma_1(a_1)=\gamma_2(a_2)$. Suppose that both $\gamma_i$ extend to $[a_i, b_i]$ such that $\gamma_i(b_i)=0$. 

If $\theta(\gamma_1)-\theta(\gamma_2)\ge \theta_0$, then the images $\gamma_1((a_1, b_1))$ and $\gamma_2((a_2,b_2))$ intersect in at least $ \ceil{\frac{\theta_0}{2\pi}}-1$ points.
\end{lemma}

\begin{proof}
    We may assume that the images intersect only in finitely many points. In particular, for small radii $r>0,$ the curves $\gamma_1$ and $\gamma_2$ do not intersect within the disks $D_r$ around the origin. Furthermore, by continuity, for $r>0$ small, both $\gamma_i$ enter the disk $D_r$ only once and then remain within the disk. 

    Since $\gamma_1(a_1)=\gamma_2(a_2)$, $\gamma_2^{-1} * \gamma_1$ is a curve that approaches the origin at both ends. Define a new curve $\Bar{\gamma}$ in $\R^2 \setminus \{ 0 \}$ by restricting $\gamma_2^{-1} * \gamma_1$ to the part outside of $D_r.$ For $\varepsilon>0$ small, we may pick $r>0$ small enough so that $\theta(\Bar{\gamma})\ge\theta_0-\varepsilon.$
    
    We then close $\Bar{\gamma}$ by connecting the endpoints along the circle of radius $r>0$ in the direction of positive winding. The resulting curve is a closed curve $c$ in $\R^2\setminus \{0\}$ with $\theta(c) \ge \theta_0-\varepsilon$. 
    
    As $c$ is closed, the winding angle $\theta(c)$ is $2\pi$ times the topological winding number of $c$ around the origin, i.e., $\theta(c) \geq 2 \pi \ceil{\frac{\theta_0-\varepsilon}{2 \pi}}.$ In particular, the winding number is at least $\ceil{\frac{\theta_0-\varepsilon}{2 \pi}} = \ceil{\frac{\theta_0}{2 \pi}},$ provided $\varepsilon>0$ is chosen small enough.

    Since the $\gamma_i$ have no self-intersections, all self-intersections of $c$ look locally like two segments of $c$ that intersect. Indeed, if three segments intersected in a point, at least two of them would have to come from the same $\gamma_i$. 

    The curve $c$ separates the plane into disjoint domains. Within each domain, the winding number of $c$ is constant. Note that $c$ has no doubled segments, since it has only finitely many self-intersections. Thus, going from one domain to the next changes the winding number of $c$ around these domains by exactly one.

    Now, if the winding number of $c$ around some point is $k \geq 1$, $c$ must have at least $k-1$ self-intersections. To see this, label each self-intersection with the lowest winding number of $c$ with respect to the adjacent regions. Then, for each $l \in \{0, ..., k-2\}$, there is at least one self-intersection with label $l.$ Therefore $c$ has at least $\ceil{\frac{\theta_0}{2\pi}}-1$ self-intersections.

    By construction, $c$ does not self-intersect along the pasted-in segment of the circle around the origin, so any self-intersection of $c$ comes from the original $\gamma_i$. Since the $\gamma_i$ have no self-intersections, each self-intersection of $c$ is an intersection of $\gamma_1$ and $\gamma_2$.
\end{proof}

\section{Estimation of the angle ODE}
\label{SectionODEestimates}

As explained in the introduction, we are interested in the behaviour of the linearized Einstein ODE along the cone solution to describe the rotational behaviour of the trajectories around the cone solution. Note that if
\begin{align*}
X(h) = |X(h)| \begin{pmatrix} \cos{\varphi(h)} \\ \sin{\varphi(h)}\end{pmatrix}
\end{align*}
satisfies the linearized Einstein ODE \eqref{LinearizationAlongConeSolution},
\begin{align*}
\frac{d}{ds}X(h(s)) = \begin{pmatrix}
0 & - \frac{2(n-1)}{n^2} \\
1 & - \frac{n-1}{n} h(s)
    \end{pmatrix} X(h(s)),
\end{align*}
then $\varphi$ solves the associated angle ODE $\varphi{'} = \frac{1}{|X|^2} \det ( X \wedge X{'}),$ where \\ $\det(X \wedge X{'})$ is the determinant of the matrix with columns $X$ and $X{'}.$ 

Explicitly,
\begin{align}
    \label{AngleODEeq1}
    \frac{d}{ds}\varphi(h(s)) = & \ \cos^2(\varphi) + \frac{2(n-1)}{n^2}\sin^2(\varphi) - \frac{n-1}{n}h(s)\sin(\varphi)\cos(\varphi)\\
    \label{AngleODEeq2}
    = & \ \left(\cos(\varphi)-\frac{n-1}{2n}h(s)\sin(\varphi)\right)^2 \\
    \nonumber & \ + \frac{2(n-1)}{n^2} \left(1-\frac{n-1}{8}h(s)^2 \right)\sin^2(\varphi)\\
    \label{AngleODEeq3}
    = & \left( 1-\frac{n-1}{8}h(s)^2 \right)\cos^2(\varphi) \\
\nonumber    & \  + \frac{2(n-1)}{n^2}\left(\sin(\varphi)  -\frac{n}{4}h(s)\cos(\varphi)\right)^2,
\end{align}
where $h(s)=-\tanh(s/n)$. \vspace{2mm}

For the construction of Einstein metrics we require an explicit estimate on the angle ODE. We achieve this for $n=9$ by constructing an explicit barrier solution $\varphi_{bar}.$

\begin{remark}\label{MonotoneAngleRemark}
    \normalfont
     If $n \leq 9,$ then $\varphi$ is non-decreasing by \eqref{AngleODEeq2}. If furthermore $n\le 8$, then $\frac{d}{ds}\varphi(h(s)) \ge c > 0$.
\end{remark}

\begin{proposition}\label{AngleBarrier}
    Let $n=9$ and let $\varphi_{bar} \colon \R \to \R$ be the solution of \eqref{AngleODEeq1} with $\varphi_{bar}(0)=\frac{3}{2}\pi$. 
    
    Then $\lim\limits_{s\to -\infty}\varphi_{bar}(h(s)) < \arctan(\frac{9}{4})$.
\end{proposition}
In the proof we frequently use the following basic observation. If 
\begin{align*}
    X=|X| \begin{pmatrix}
    \cos ( \varphi ) \\
    \sin ( \varphi )
\end{pmatrix}
\text{ solves } X' = \begin{pmatrix}
    a & b \\
    c & d
\end{pmatrix} X,
\end{align*}
then $\varphi$ satisfies the associated ODE
\begin{align*}
    \varphi' = -b \sin^2(\varphi) + c \cos^2 ( \varphi ) + (d-a) \sin( \varphi ) \cos( \varphi ).
\end{align*}
In particular, we can obtain estimates on $\varphi$ by solving the associated linear ODE. 

Furthermore, we use the substitution 
\begin{align*}
    \frac{d}{dh} = - \frac{n}{1-h^2} \frac{d}{ds}. \vspace{2mm}
\end{align*}

\begin{proof}
The proof proceeds in four steps, as we require a different form of the angle ODE \eqref{AngleODEeq1} - \eqref{AngleODEeq3} for the comparison argument depending on the values of the barrier solution $\varphi_{bar}.$ Note that we prescribe $\varphi_{bar}$ at $s=h(s)=0$ and prove estimates as $s \to - \infty,$ i.e., $h(s) \to 1.$ \vspace{2mm}

\textit{Step 1.} $\varphi_{bar}(h=\frac{1}{4}) \le \pi + \arctan(\frac{9}{4}).$ \vspace{2mm}

Equation \eqref{AngleODEeq2} implies that $\frac{d}{ds}\varphi_{bar}(h(s)) \ge \frac{16}{81}(1-h(s)^2)\sin^2(\varphi_{bar})$. We compare with the linear ODE associated to the right hand side, 
\begin{align*}
\frac{d}{ds}A(h(s)) = 
\begin{pmatrix}
0 & -\frac{16}{81}(1-h(s)^2) \\
0 & 0
\end{pmatrix} 
A, \ \ A(0)=
\begin{pmatrix}
0 \\ 
-1
\end{pmatrix}.
\end{align*}
Thus we obtain
\begin{align*}
  \frac{d}{dh}A(h) = 
  \begin{pmatrix}
  0 & \frac{16}{9} \\
  0 & 0
\end{pmatrix} 
A, \ A(0)=\begin{pmatrix}0 \\
-1
\end{pmatrix}.
\end{align*}
The explicit solution $A(h) = 
\begin{pmatrix}
-\frac{16}{9}h \\
-1
\end{pmatrix}$ 
satisfies $A(h=\frac{1}{4}) = 
\begin{pmatrix}
-\frac{4}{9} \\
-1
\end{pmatrix}$. Thus for the associated angle $\varphi_A$ we have $\varphi_A(h=\frac{1}{4})=\pi + \arctan(\frac{9}{4}).$ \vspace{2mm}

\textit{Step 2.} $\varphi_{bar}(h=\frac{1}{2}) \le \pi$. \vspace{2mm}

Equation \eqref{AngleODEeq3} shows that we have
\begin{align*}
\frac{d}{ds}\varphi_{bar}(h(s)) \ge (1- h(s)^2) \cos^2(\varphi_{bar}), \ \
\varphi_{bar}(h=\frac{1}{4}) \le \pi + \arctan(\frac{9}{4}).  
\end{align*}
Thus we may again compare with the associated linear ODE
\begin{align*}
\frac{d}{ds}B(h(s)) = 
\begin{pmatrix}
0 & 0 \\
1-h(s)^2 & 0
\end{pmatrix}
B, \ B(h=\frac{1}{4})= 
\begin{pmatrix}
-\frac{4}{9} \\
-1
\end{pmatrix}.    
\end{align*}
We substitute again to see that this is equivalently
\begin{align*}
 \frac{d}{dh}B(h) = 
 \begin{pmatrix}
 0 & 0 \\
 -9 & 0
 \end{pmatrix}
B, \ B(h=\frac{1}{4})= 
\begin{pmatrix}
-\frac{4}{9} \\
-1\end{pmatrix},
\end{align*}
which has the explicit solution $B(h)=
\begin{pmatrix}
-\frac{4}{9} \\
-1 + 4(h-\frac{1}{4})
\end{pmatrix}$. In particular, $B(h= \frac{1}{2})= 
\begin{pmatrix}
-\frac{4}{9} \\
0\end{pmatrix}$
and the associated angle $\varphi_B$ satisfies $\varphi_B(h=\frac{1}{2})=\pi$. \vspace{2mm}

\textit{Step 3.} $\varphi_{bar}(h(s_0 - \frac{9}{4}\frac{\pi}{2})) \le \frac{\pi}{2}$, where $s_0$ is determined by $h(s_0)=\frac{1}{2}$. \vspace{2mm}

Within the region $\varphi\in [\frac{\pi}{2}, \pi]$, we have $\sin(\varphi) \cos(\varphi)\le 0$. Thus \eqref{AngleODEeq1} shows that
\begin{align*}
  \frac{d}{ds}\varphi_{bar}(h(s)) \ge \cos^2(\varphi_{bar}) + \frac{16}{81}\sin^2(\varphi_{bar}), \ \varphi_{bar}(h(s_0))\le \pi
\end{align*}
and the associated comparison ODE is 
\begin{align*}
\frac{d}{ds}C(h(s)) =  \begin{pmatrix}
0 & -\frac{16}{81}\\ 1 & 0
\end{pmatrix}
C, \ C(h(s_0))=
\begin{pmatrix}-1 \\ 0
\end{pmatrix}.    
\end{align*}
This linear ODE has the explicit solution 
\begin{align*}
  C(h(s = s_0+\tau )) = 
  \begin{pmatrix}
  -\cos(\frac{4}{9}\tau) \\
  -\frac{9}{4}\sin(\frac{4}{9}\tau))
  \end{pmatrix}.  
\end{align*}
In particular, $C(h(s_0-\frac{9}{4}\frac{\pi}{2}))=\begin{pmatrix}0 \\ \frac{9}{4}\end{pmatrix}$ with associated angle $\varphi_C(h(s_0-\frac{9}{4}\frac{\pi}{2}))=\frac{\pi}{2}$. \vspace{2mm}

\textit{Step 4.} $\varphi_{bar}(h(s_0-\frac{9}{4}\frac{\pi}{2}) + \frac{1}{4}) \le \arctan(\frac{9}{4})$.

As in {\em Step 1} it follows that we may compare with 
\begin{align*}
  \frac{d}{dh}A(h) = 
  \begin{pmatrix}
  0 & \frac{16}{9} \\
  0 & 0
\end{pmatrix}
A, \ A(h(s_0-\frac{9}{4}\frac{\pi}{2})) = 
\begin{pmatrix}
0
\\
1
\end{pmatrix},
\end{align*}
which has the explicit solution 
\begin{align*}
A(h(s_0-\frac{9}{4}\frac{\pi}{2}) + \tau) = 
\begin{pmatrix}
\frac{16}{9}\tau \\
1
\end{pmatrix}. 
\end{align*}
Thus $A(h(s_0-\frac{9}{4}\frac{\pi}{2}) + \frac{1}{4})=
\begin{pmatrix}
\frac{4}{9} \\
1
\end{pmatrix}$ and the associated angle satisfies $\varphi_A(h(s_0-\frac{9}{4}\frac{\pi}{2}) + \frac{1}{4}) = \arctan(\frac{9}{4})$. 

It remains to remark that $h(s_0-\frac{9}{4}\frac{\pi}{2}) + \frac{1}{4} < 1$. This is an explicit calculation. Note that $s_0 = 9\operatorname{artanh}(-\frac{1}{2}) \approx -4.94$ and $h(s_0-\frac{9}{4}\frac{\pi}{2}) \approx 0.73 < \frac{3}{4}$.

Finally, we obtain the strict inequality $\lim\limits_{s\to -\infty}\varphi_{bar} \le \varphi_{bar}(h(s_0-\frac{9}{4}\frac{\pi}{2})) < \arctan(\frac{9}{4})$ by noting that our comparisons are not globally sharp. Alternatively, one may proceed as in {\em Step 5'} of the proof of Lemma \ref{PBDLemma} below.
\end{proof}

In the proof of Theorem \ref{EinsteinMetricsOnSpheresMainTheorem} in section \ref{SectionProofOfMainTheorem}, we actually use the following quantitative refinement of Proposition \ref{AngleBarrier}:

\begin{lemma}
\label{PBDLemma}
Let $n=9$. There exist $\varepsilon>0$ and $\delta_0>0$ such that for all $\delta \in [0, \delta_0)$ the solution $\varphi_{bar}^{\delta} \colon \R \to \R$ of the initial value problem
\begin{align}
\label{DeltaAngleODE}
    \frac{d}{ds}\varphi(h(s)) & = \cos^2(\varphi) + \frac{2(n-1)}{n^2}\sin^2(\varphi) - \frac{n-1}{n}h(s)\sin(\varphi)\cos(\varphi) - \delta, \\
    \nonumber
    \varphi_{bar}^{\delta}(0) & =\frac{3}{2}\pi + \delta,
\end{align}
where $h(s)=-\tanh(s/n),$ satisfies
\begin{align*}
\varphi_{bar}^{\delta}(h(s)) < \arctan(\frac{9}{4}) - \varepsilon
\end{align*}
for all $s \in \R$ with $h(s)>0.995$. 

In particular, $\lim\limits_{s \to - \infty} \varphi_{bar}^{\delta}(h(s)) < \arctan(\frac{9}{4}) - \varepsilon.$
\end{lemma}

\begin{proof}
The proof follows the ideas of the proof of Proposition \ref{AngleBarrier}.  \vspace{2mm}

\textit{Step 1'}. $\varphi_{bar}^{\delta}(h=\frac{1 + \frac{9}{4}\tan(\delta)}{4 (1 - 17\delta)})\le \pi + \arctan(\frac{9}{4})$. \vspace{2mm}

As in {\em Step 1} of the proof of Proposition \ref{AngleBarrier}, we conclude
\begin{align*}
    \frac{d}{dh}\varphi_{bar}^{\delta} \le -\frac{16}{9} \sin^2(\varphi_{bar}^{\delta}) + \frac{9}{1-h^2}\delta.
\end{align*}

Note that $4^2+9^2=97$ and thus for $\varphi_{bar}^{\delta} \in [\pi+ \arctan(\frac{9}{4}), \frac{3}{2}\pi + \delta]$ we have $\sin^2(\varphi_{bar}^{\delta})\in[\frac{81}{97}, 1]$ and hence $\frac{16}{9}\sin^2(\varphi_{bar}^{\delta}) > 1$. Furthermore, since $\frac{9}{1-\frac{4}{9}} = \frac{81}{5}<17<17 \cdot \frac{16}{9} \sin^2( \varphi_{bar}^{\delta})$, we obtain for $h\le \frac{2}{3}$ the estimate
\begin{align*}
    \frac{d}{dh}\varphi_{bar}^{\delta} \leq -\frac{16}{9} \sin^2(\varphi_{bar}^{\delta}) + \frac{9}{1-h^2}\delta     \le - \frac{16}{9} (1 - 17 \delta) \sin^2(\varphi_{bar}^{\delta}).
\end{align*}

From here we proceed as in {\em Step 1} of the proof of Proposition \ref{AngleBarrier}, using the initial condition $\begin{pmatrix}
\tan \delta \\
-1
\end{pmatrix},$
to conclude that 
$\pbd(h=\frac{1 + \frac{9}{4}\tan(\delta)}{4 (1 - 17\delta)}) \le \pi + \arctan(\frac{9}{4})$. \vspace{2mm}

\textit{Step 2'.} $\pbd(h=\frac{2 + \frac{9}{4}\tan(\delta)}{4(1-17\delta)}) \le \pi.$ \vspace{2mm}

As in {\em Step 2} of the proof of Proposition \ref{AngleBarrier}, we obtain
\begin{align*}
    \frac{d}{dh}\pbd \le - 9 \cos^2(\pbd) + \frac{9}{1-h^2}\delta.
\end{align*}

Now we proceed as in {\em Step 1'}. On the interval $[\pi, \pi+\arctan(\frac{9}{4})]$ we have $\frac{16}{97} \le \cos^2 (\pbd)\le 1$ and hence $9\cos^2(\pbd) \ge 1.$ Therefore, for $h\le \frac{2}{3}$, we obtain the estimate
\begin{align*}
    \frac{d}{dh}\pbd \le - 9 (1-17\delta) \cos^2(\pbd)
\end{align*}
and thus $\pbd(h=\frac{1 + \frac{9}{4}\tan(\delta)}{4(1-17\delta)} + \frac{1}{4(1-17\delta)}) \le \pi$ by comparison. \vspace{2mm}

\textit{Step 3'.}  $\pbd(h(s_0^{\delta} - \frac{9}{4}\frac{\pi}{2}\frac{1}{1-\frac{81}{16}\delta}))\le \frac{\pi}{2},$ where $s_0^{\delta}$ is determined by $h(s_0^{\delta})=\frac{2 + \frac{9}{4}\tan(\delta)}{4(1-17\delta)}$. \vspace{2mm}

Proceeding as in {\em Step 3} of the proof of Proposition \ref{AngleBarrier}, we derive 
\begin{align*}
    \frac{d}{ds}\pbd & \ge \cos^2(\pbd) + \frac{16}{81}\sin^2(\pbd) - \delta \\
    &  \ge \left(\cos^2(\pbd) + \frac{16}{81}\sin^2(\pbd)\right) (1 - \frac{81}{16} \delta)
\end{align*}
and conclude $\pbd(h(s_0^{\delta} - \frac{9}{4}\frac{\pi}{2}\frac{1}{1-\frac{81}{16}\delta}))\le \frac{\pi}{2}$ by comparison. \vspace{2mm}

\textit{Step 4'.} There exists $\delta_0 >0$ such that for all $\delta \in [0, \delta_0)$ we have 
\begin{align*}
    \pbd(h(s_0^{\delta} - \frac{9}{4}\frac{\pi}{2}\frac{1}{1-\frac{81}{16}\delta}) + \frac{1}{4(1-455\delta)}) \le \arctan(\frac{9}{4}).
\end{align*}

Note that there is $\delta_0 >0$ so that for all $\delta \in [0, \delta_0)$ we have $h(s_0^{\delta} - \frac{9}{4}\frac{\pi}{2}\frac{1}{1-\frac{81}{16}\delta}) + \frac{1}{4(1-455\delta)} < 0.99.$ Indeed, by continuity, it suffices to note that for $\delta = 0$ we have $h(s_0- \frac{9}{4}\frac{\pi}{2}) + \frac{1}{4} \approx 0.986.$

Now we can proceed as in {\em Step 1'}, except that we replace the estimate $h \leq \frac{2}{3}$ by the estimate $h \leq 0.99.$ \vspace{2mm}
    
\textit{Step 5'.} There are $\varepsilon > 0$ and $\delta_0>0$ such that for all $\delta \in [0,\delta_0)$ and $h>0.995$ we have 
\begin{align*}
 \pbd(h) \le \arctan(\frac{9}{4})-\varepsilon.
\end{align*}
The estimate in {\em Step 4'} shows that there is $\varepsilon>0$ such that $\pbd(0.995) \leq \arctan(\frac{9}{4})-\varepsilon$ for $\delta \in [0,\delta_0)$. Let $\Psi(\phi,h)$ denote the right hand side of \eqref{DeltaAngleODE}. After possibly shrinking $\delta_0>0$, note that for all $h \in (-1,1)$ and $\delta \in [0, \delta_0)$ we have
\begin{align*}
 \Psi(\phi=\arctan(\frac{9}{4})-\varepsilon,h) 
    &= [\cos(\phi)-\frac{4}{9} h \sin(\phi)]^2 + \frac{16}{81}(1-h^2)\sin^2(\phi) - \delta \\
    &\ge [\cos(\phi)-\frac{4}{9}h \sin(\phi)]^2 - \delta\\
    &\ge [\cos(\phi)-\frac{4}{9}\sin(\phi)]^2 - \delta\\
    &=\frac{97}{81}\sin^2(\varepsilon) - \delta\\
    &\ge 0.
\end{align*}
In particular, it follows that $\pbd(h) \le \arctan(\frac{9}{4})-\varepsilon$ for all $h>0.995$ and $\delta<\delta_0$, as claimed.
\end{proof}

\section{Proof of the main theorem}
\label{SectionProofOfMainTheorem}

Recall that $SO(d_1+1) \times SO(d_2+1)$-invariant Einstein metrics on $S^{d_1+d_2+1}$ correspond to points in $M_1^+ \cap \overline{M_2^+} \cap \{H=0\}$ due to Lemma \ref{IntersectionsAreSolutions}. Thus, for a given pair $(d_1,d_2)=(2,7),$ $(3,6),$ $(4,5)$, the existence of a non-round $SO(d_1+1) \times SO(d_2+1)$-invariant Einstein metric on $S^{10}$ is equivalent to the existence of at least two points in $M_1^+ \cap \overline{M_2^+} \cap \{H=0\}.$ To produce these intersections, we will use Lemma \ref{ThetaBoundsIntersections} on the curves $M_1^+ \cap \{H=0\}$ and $\overline{M_2^+} \cap \{H=0\}$.\vspace{2mm}

\textit{Proof of Theorem \ref{EinsteinMetricsOnSpheresMainTheorem}.} Fix $a>0$ and a parametrization $\Tilde{\gamma}^i_0 \colon [-a,0) \to \{ H = 0 \}$ of $M^+_i\cap \{H=0\}$ for $i=1,2$ such that $\Tilde{\gamma}^i_0$ extends continuously to $[-a,0]$ with value $\Tilde{\gamma}^i_0(0) = (0,0,0)$ in $(Z,\Delta,H)$-coordinates. This is possible since trajectories converge to the cone solution $\cone(h)=(0,0,h)$, see Remark \ref{InvariantSets}.

Note that $\Tilde{\gamma}_0^i(-a)\in \partial\mathcal{S},$ since, by Corollary \ref{UniqueIntersectionWithBoundary}, there is a trajectory of the Einstein ODE \eqref{EinsteinODEinZDeltaH} emanating from $p_i^+$ that lies entirely in the boundary $\partial \mathcal{S}$. Moreover, this trajectory remains in the same quadrant of $M_i^+ \cap \{ H = h \}$ as $p_i^+$ since it cannot intersect the explicit solutions $q_i^+(h)$ of Proposition \ref{SpecialSolutions} which occupy $\partial\mathcal{S} \cap \{Z=0\}$ and has the rotational behaviour described in Proposition \ref{QuadrantRotation}. In particular, $\sign{Z(\Tilde{\gamma}^i_0(-a))} = \sign{Z(p_i^+)}.$

Therefore we can extend $\Tilde{\gamma}^i_0$ to a curve $\gamma^i_0 \colon [-b,0) \to \{ H = 0 \}$ by connecting $\Tilde{\gamma}^i_0(-a) \in \partial\mathcal{S}$ to $q_i^+(0)$ along the boundary of $\mathcal{S}$, staying in the half-plane $\{\sign{Z}=\sign{Z(p_i^+)}\}$ on $(-b,-a]$, see Figure \ref{fig:wrapfig} with $h=0$. 
Note that we can ensure that $\Tilde{\gamma}^i_0$ does not have self-intersections since $\Tilde{\gamma}^i_0$ only meets $\partial\mathcal{S}$ in the single point $\Tilde{\gamma}_0^i(-a)$ by Corollary \ref{UniqueIntersectionWithBoundary}.

Recalling $\overline{(Z, \Delta, H)} = (Z, -\Delta, H)$, as in Definition \ref{defBarInvolution}, the two curves $\gamma^1_0$ and $\overline{\gamma^2_0}$ have the same starting point $q_1^+(0)=\overline{q_2^+(0)}$ and we are in the position of Lemma \ref{ThetaBoundsIntersections}. Theorem \ref{EinsteinMetricsOnSpheresMainTheorem} follows if we can prove $\ceil{\frac{|\theta(\gamma^1_0)-\theta(\overline{\gamma^2_0})|}{2\pi}}-1>1$, i.e., $\theta(\gamma^1_0)-\theta(\overline{\gamma^2_0})=\theta(\gamma^1_0)+\theta(\gamma^2_0)>4\pi$, in the sense of Definition \ref{derivativecontrolswinding}. 

We shall prove $\theta(\gamma^1_0)>2\pi$, with the argument that $\theta(\gamma_0^2)>2\pi$ being completely analogous. Going forward, we set $\gamma_0=\gamma_0^1$. \vspace{2mm}

\begin{wrapfigure}{r}{0.25\textwidth}
\includegraphics[width=0.9\linewidth]{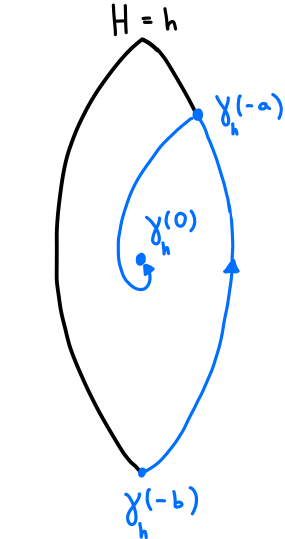}
    \caption{\,\\ The curve $\gamma_h$}
\label{fig:wrapfig}
\end{wrapfigure}

For $h \in (-1,1)$, we define curves $\gamma_h:[-b, 0) \to (M_1^+\cup \partial \mathcal{S})\cap \{H=h\}$ in the following way: For $\tau \in [-b, 0)$, let $\gamma_h(\tau)$ be the unique intersection point of the trajectory of the Einstein ODE passing through $\gamma_0(\tau)$ and the slice $\{H=h\}$. Note that then, for each $\tau,$ the curve $h\mapsto \gamma_h(\tau)$ is a reparametrization of a trajectory of the Einstein ODE. 

Identifying each slice $\{H=h\}$ with $\R^2$, we may write 
\begin{align*}
\hspace{-0.33\textwidth}
\gamma_h(\tau)=|\gamma_h(\tau)|
\begin{pmatrix}
\cos \varphi(h, \tau) \\ 
\sin \varphi(h, \tau)
\end{pmatrix}    
\end{align*}
for a function $\varphi$ whose indeterminacy is fixed by setting $\varphi(0, -b) = -\frac{\pi}{2}.$ Note that hence $\varphi(h, -b) = -\frac{\pi}{2}$ for all $h \in (-1,1)$ by the explicit singular solutions of Proposition \ref{SpecialSolutions}. 

By the explicit characterization of $\theta$ in definition \ref{WindingAngle}, it remains to show that $\theta(\gamma_0)>2\pi$. This follows immediately from the following proposition:

\begin{proposition}
\label{NonLinearComparison}
For $\delta>0, \varepsilon>0$ sufficiently small, there is $\tau_0 > 0$ such that for all $\tau \in [-\tau_0, 0)$ and $h \in [0, 1-\varepsilon]$ we have
\begin{align*}
    \varphi(h, \tau) > \pbd(h),
\end{align*}
where $\pbd$ is as defined in Lemma \ref{PBDLemma}.
\end{proposition}

In particular, it follows that for all $\tau \ge -\tau_0$, 
\begin{align*}
\theta({\gamma_0}_{|[-b, \tau]}) 
 = \varphi(0, \tau) - \varphi(0,-b) \ge \pbd(0) - \varphi(0, -b) = \frac{3}{2}\pi + \delta + \frac{1}{2}\pi = 2\pi + \delta.
\end{align*}
This completes the proof of Theorem \ref{EinsteinMetricsOnSpheresMainTheorem}. $\hfill \Box$

\begin{proof}[Proof of Proposition \ref{NonLinearComparison}] 
Recall that $\pbd$ solves the ODE
\begin{align}
    \tag{\ref{DeltaAngleODE}} \frac{d}{ds}\pbd(h(s)) = & \ \cos^2(\pbd) + \frac{2(n-1)}{n^2}\sin^2(\pbd) \\
\nonumber    & \ - \frac{n-1}{n}h(s)\sin(\pbd)\cos(\pbd) - \delta.
\end{align}
We define a helicoidal surface $\Psi^\delta$ in $(Z,\Delta, H)$-coordinates by setting
\begin{align*}
  \Psi^\delta(h, r) = (0,0,h) + r(\cos(\pbd(h)), \sin(\pbd(h)), 0).  
\end{align*}

The aim is to use $\Psi^\delta$ as a barrier for solutions of the Einstein ODE. Let $N = \frac{\partial\Psi}{\partial r} \times \frac{\partial \Psi}{\partial s}$ denote the forward normal of $\Psi^\delta$ and let $E$ denote the vector field defined by the Einstein ODE \eqref{EinsteinODEinZDeltaH}. It is straightforward to compute that 
\begin{align*}
    \langle N, E \rangle_{|\Psi^\delta(h,r)}
    =  \frac{1-h^2}{n} \left( \delta r + O(r^2) \right).
\end{align*}

In particular, for $r \in (0,r_0(\delta))$, we have $\langle N, E \rangle_{|\Psi^\delta(h,r)} >0.$

We choose a constant $c = c(\delta) < \mathcal{F}_{\cone}$ such that the invariant set $I_c = \{\mathcal{F} \ge c\}$ of Remark \ref{InvariantSets} is contained in the $r_0$-tube around the cone solution.

Pick $\varepsilon_{bar}>0$ as in Lemma \ref{PBDLemma}, i.e., such that $\pbd(h) < \arctan(\frac{9}{4})-\varepsilon_{bar}$ for all $\delta>0$ small, $h>0.995$. Furthermore, fix a parametrization of $\gamma^{RF}_1 \subset M_1^+ \cap \{ H=1 \}$ defined on $(t_1, t_2)$ as 
\begin{align*}
    \gamma^{RF}_1(t) = |\gamma^{RF}_1(t)|
    \begin{pmatrix}
    \cos(\varphi^{RF}(t)) \\
    \sin(\varphi^{RF}(t))
    \end{pmatrix}
\end{align*}
with the normalization $\lim\limits_{t\to t_1}\varphi^{RF}(t)\in (0, \frac{\pi}{2})$. Note that in this normalization the function
\begin{align*}
\varphi \colon  M_1^+ \cap \{-1 < H \le 1\} \to \R, \ \ p \mapsto \begin{cases} 
\varphi^{RF}(t), & \text{if } p=\gamma^{RF}(t), \\ 
\varphi(h, \tau), & \text{if } p=\gamma_h(\tau)
\end{cases}    
\end{align*}
is continuous.

Denote by $t^{RF}$ the time $\gamma^{RF}$ enters $I_c$. By continuity and the fact that $\varepsilon_{bar}$ is independent of $\delta$, we may assume that $c=c(\delta)$ was chosen close enough to $\mathcal{F}_{\cone}$ such that $|\varphi^{RF}(t^{RF}) - \lim\limits_{t\to t_2}\varphi^{RF}(t)| < \frac{\varepsilon_{bar}}{2}$.

Since $\gamma^{RF}(t^{RF})$ lies outside a compact neighborhood of the stationary point $\cone^+$, continuous dependence on initial conditions shows that for $\tau<0$ close to zero, the starting segments of the trajectories $h \mapsto \gamma_h(\tau)$ approximate $\gamma^{RF}_{|(t_1, t^{RF}]}$ in $C^0$. Let $h_{\tau}$ be the value at which $h \mapsto \gamma_h(\tau)$ intersects $I_c$. Note that $\gamma_{h_{\tau}}(\tau)$ converges to $\gamma^{RF}(t^{RF})$ as $\tau \to 0.$ In particular, we may pick $\tau_0>0$ such that $|\varphi(h_{\tau}, \tau) - \varphi^{RF}(t^{RF})| < \frac{\varepsilon_{bar}}{2}$ for $|\tau| < \tau_0$.

Since the intersection point $\gamma_{h_{\tau}}(\tau)$ approaches $\gamma^{RF}(t^{RF})$ as $\tau \to 0$, we see that $H(\gamma_{h_{\tau}}(\tau)) = h_{\tau} \to 1$ as $\tau \to 0.$ In particular, after possibly decreasing $\tau_0>0$, we may assume $h_{\tau} > 1 - \varepsilon$ for all $|\tau| < \tau_0$. \vspace{2mm}

By our choice of $c$, if $\gamma_h(\tau)$ lies within $I_c$ for some $h, \tau$ with $\varphi(h, \tau) > \pbd(h)$, then comparison with $\Psi^\delta$ shows $\varphi(h', \tau) > \pbd(h')$ for all $h'\in [0, h]$. Therefore it remains to show that $\varphi(h_{\tau}, \tau) > \pbd(h_{\tau})$ for $|\tau| < \tau_0$.

We see this as follows: Corollary \ref{TangentDirectionAtCone+} shows that $\lim\limits_{t\to t_2}\varphi^{RF}(t) = \arctan(\frac{9}{4})$ $\text{mod } \pi$ and Proposition \ref{QuadrantRotation} shows that $\lim\limits_{t\to t_2}\varphi^{RF}(t) \ge 0$. Together this implies $\lim\limits_{t\to t_2}\varphi^{RF}(t) \ge \arctan(\frac{9}{4})$. Thus, by the triangle inequality, we obtain that for all $|\tau|<\tau_0$
\begin{align*}
  \varphi(h_{\tau}, \tau) &> \lim\limits_{t\to t_2}\varphi^{RF}(t) - \frac{\varepsilon_{bar}}{2} - \frac{\varepsilon_{bar}}{2} \\
  &\ge  \arctan(\frac{9}{4}) - \varepsilon_{bar} \\
  &> \pbd(h_{\tau}),  
\end{align*}
where we used Lemma \ref{PBDLemma} in the last line and assumed that $\varepsilon < 0.005$, so that $h_{\tau} > 1-\varepsilon > 0.995$.
\end{proof}

\begin{remark}[On the distinctness of the constructed metrics]\,
\normalfont
    With the methods of the proof, we find two $SO(d_1+1)\times SO(d_2+1)$-invariant Einstein metrics on $S^{10}$ for each pair $d_1+d_2=9, d_i\ge 2$. The metrics obtained for $(d_1, d_2)$ are clearly isometric to those obtained for $(d_2, d_1)$, so we may restrict to $d_1 \le d_2$. It is also clear that for any choice of $(d_1, d_2)$ one of the two metrics is the round metric on $S^{10}$. We claim that the non-round metrics $g_{(d_1, d_2)}$ for each pair $(d_1, d_2)$ with $d_1 \le d_2$ are distinct.
    
    Otherwise, we may assume there is an isometry between $g_{(d_1, d_2)}$ and $g_{(\Bar{d}_1, \Bar{d}_2)}$. By pullback, this implies that $g_{(d_1, d_2)}$ is invariant under a connected Lie group $G$ containing subgroups isomorphic to  $SO(d_1+1)\times SO(d_2+1)$ and $SO(\Bar{d}_1+1)\times SO(\Bar{d}_2+1)$ that act with the standard cohomogeneity one structure. The action of $G$ is now homogeneous: Indeed, a $G$-principal orbit contains both $S^{d_1}\times S^{d_2}$ and $S^{\Bar{d}_1}\times S^{\Bar{d}_2}$. Since these are topologically distinct closed $9$-manifolds, the dimension of a $G$-principal orbit must be strictly larger than $9$. Therefore the orbit space $S^{10}/G$ must be a point, i.e., the action of $G$ must be transitive and the metric must be homogeneous. By work of Ziller \cite{ZillerHomogeneousEinsteinMetrics}, the only homogeneous Einstein metric on $S^{10}$ is the round one, giving us the desired contradiction. Thus we see that the metrics $g_{(2,7)}, g_{(3,6)}, g_{(4,5)}$ and $g_{\text{round}}$ are pairwise non-isometric Einstein metrics on $S^{10}$.
\end{remark}

\begin{remark}[Reconstruction of B\"ohm's metrics]\label{boehmrecovery}
\,

\normalfont
    In order to construct the infinite families of Einstein metrics on $S^{n+1}$ for $n+1\in \{ 5, \ldots, 9 \}$ first described by B\"ohm in \cite{BohmInhomEinstein}, one may proceed like this:

    From the linearization at the cone solution in Remark \ref{RemarkLinearizationAlongTheConeSolution}, one finds that the Ricci-flat trajectory $\gamma^{RF}$ spirals into the cone. In particular, $\lim \varphi^{RF}=\infty$. By Remark \ref{MonotoneAngleRemark}, one may then use barrier surfaces $\Psi_\psi(h,r)=(0,0,h)+r(\cos\psi, \sin\psi, 0)$ with any constant $\psi$ to see that $\theta(\gamma^i_0) \ge \psi + \frac{\pi}{2}$ for all $\psi \in \R$ in the sense of definition \ref{derivativecontrolswinding}. Therefore $\gamma^1_0$ and $\conj{\gamma^2_0}$ intersect infinitely often by Lemma \ref{ThetaBoundsIntersections}, giving us infinitely many $SO(d_1+1)\times SO(d_2+1)$-invariant Einstein metrics on $S^{d_1+d_2+1}$ for each pair $(d_1, d_2)$ with $d_1, d_2\ge 2, d_1+d_2\le 8$.
\end{remark}

\begin{remark}[On the degenerate case $d_1=1$]\label{d1Remark}\,

\normalfont
    Throughout this paper, we have assumed $d_1, d_2\ge 2$. For $d_1=1$, one sees that the conservation law $\mathcal{S}_1=0$ can be used to decouple the equations \eqref{EinsteinODEinHYDelta} for $(H, \Delta)$ from the $Y_i$. Analysis of this decoupled system then shows that the round metric is the only Einstein metric in this setup. 
\end{remark}

\begin{remark}[On numerics and spheres of higher dimensions]\,
\label{NumericalSolutions}

\normalfont
    Conceptually, the methods of our proof are not restricted to the case $n=9$, i.e., to finding Einstein metrics on $S^{10}$. However, for spheres of higher dimensions (with the same cohomogeneity one structure), one finds numerically that the required analogue of Proposition \ref{AngleBarrier} does not hold. While this does not rule out the existence of Einstein metrics of the given type, it does mean that these metrics do not come from the linearized behaviour around the cone solution. 
    
    This is supported by other numerical findings: In private communication, C. B\"ohm explained to us that he numerically discovered Einstein metrics on $S^{11}$ for some, but not all, values of $(d_1, d_2).$ In the particular case of $d_1=d_2=5$, Dancer-Hall-Wang \cite{DHWShrinkingSolitons} also performed a numerical study and did not detect any Einstein metrics. Since the linearization around the cone solution depends only on $n=d_1+d_2$, this means that any metrics that may exist on $S^{11}$ cannot be detected from the behaviour of the linearization.\vspace{2mm}
    
    In our own numerical study, we found a non-round Einstein metric on $S^{11}$ for each of the pairs $(d_1, d_2) = (2,8)$ and $(3,7)$, but none for $(4,6)$ and $(5,5)$.
    
    On $S^{10}$, we numerically found exactly one non-round Einstein metric for each pair $(d_1, d_2)$, matching with the metrics of Theorem \ref{EinsteinMetricsOnSpheresMainTheorem}, leading us to conjecture that these are all $SO(d_1) \times SO(d_2)$-invariant Einstein metrics on $S^{10}.$ Further numerical investigations indicated that none of the other cohomogeneity one structures on $S^{10}$ supports a non-round Einstein metric, leading to the conjecture

    \begin{conjecture*}
        The metrics of Theorem \ref{EinsteinMetricsOnSpheresMainTheorem} are all cohomogeneity one Einstein metrics on $S^{10}.$
    \end{conjecture*}
    
    Furthermore, we find a nontrivial Einstein metric on $S^{d_1+1} \times S^{d_2}$ only for the pair $(d_1,d_2)=(2,7).$ \vspace{2mm}

    Figures \ref{45Plots} - \ref{27Plots} provide plots of the $\{ H = 0 \}$-slices for the different $(d_1,d_2)$-systems for $n+1=d_1+d_2+1=10.$ 
 
    Intersections of $M_1^+$ with $M_2^-$ correspond to Einstein metrics on $S^{10}.$ The outer intersection point represents the round sphere. Any other intersection point corresponds to a non-trivial Einstein metric. By symmetry, we find corresponding intersection points of $M_2^+$ with $M_1^-$, which give rise to isometric metrics. 

    Intersections of $M_i^+$ with $M_i^-$ correspond to Einstein metrics on $S^{d_i+1} \times S^{d_j}.$ We always see two of these intersections, corresponding to the products of round spheres. Apart from these, only the case $(2,7)$ appears to have a new example, which stems from the fact that in this case $M_1^+$ and $M_2^+$ make an angle of $0$ instead of $\pi$ at the cone point due to the behaviour of the Ricci-flat subsystem.
\end{remark}

\setcounter{figure}{3}
\begin{figure}[h]
    \centering
    \begin{overpic}[width=0.45\textwidth]{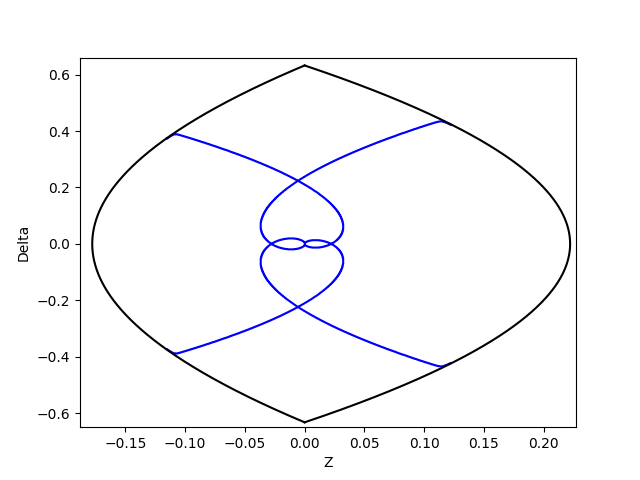}
\put(690,570){$\scriptstyle M_1^+$}
\put(190,560){$\scriptstyle M_2^-$}
\put(190,150){$\scriptstyle M_2^+$}
\put(690,140){$\scriptstyle M_1^-$}
\end{overpic}
\begin{overpic}[width=0.45\textwidth]{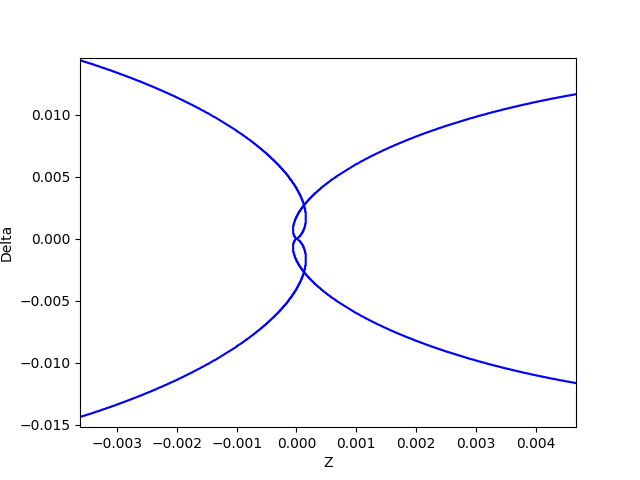}
\put(800,540){$\scriptstyle M_2^+$}
\put(135,580){$\scriptstyle M_1^-$}
\put(135,150){$\scriptstyle M_1^+$}
\put(800,185){$\scriptstyle M_2^-$}
\end{overpic}
    \caption{The $\{ H=0 \}$-slice of the $(4,5)$-system.}
    \label{45Plots}
\end{figure}
\begin{figure}[h]
    \centering
    \begin{overpic}[width=0.45\textwidth]{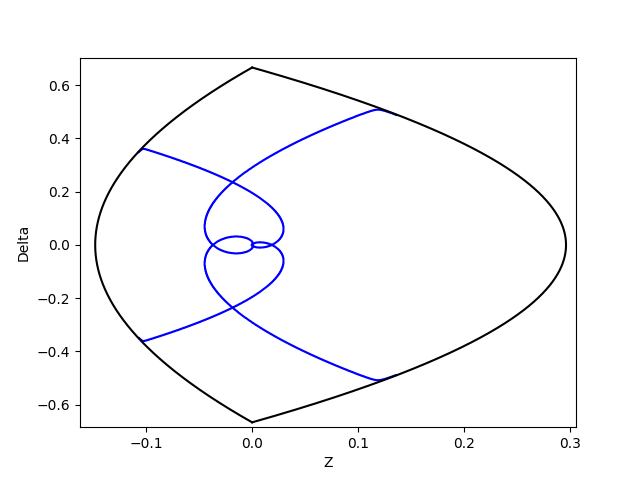}
\put(575,595){$\scriptstyle M_1^+$}
\put(135,525){$\scriptstyle M_2^-$}
\put(135,175){$\scriptstyle M_2^+$}
\put(575,110){$\scriptstyle M_1^-$}
\end{overpic}
\begin{overpic}[width=0.45\textwidth]{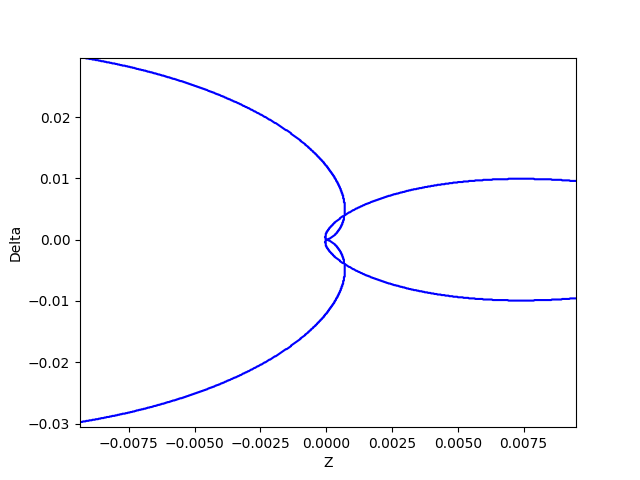}
\put(805,495){$\scriptstyle M_2^+$}
\put(130,600){$\scriptstyle M_1^-$}
\put(130,130){$\scriptstyle M_1^+$}
\put(805,235){$\scriptstyle M_2^-$}
\end{overpic}
    \caption{The $\{ H=0 \}$-slice of the $(3,6)$-system.}
    \label{36Plots}
\end{figure}
\begin{figure}[h]    
    \centering
    \begin{overpic}[width=0.45\textwidth]{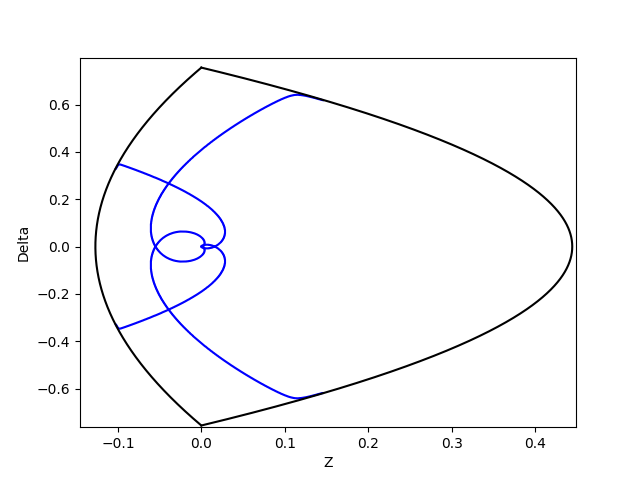}
\put(410,540){$\scriptstyle M_1^+$}
\put(128,545){$\scriptstyle M_2^-$}
\put(128,145){$\scriptstyle M_2^+$}
\put(415,155){$\scriptstyle M_1^-$}
\end{overpic}
\begin{overpic}[width=0.45\textwidth]{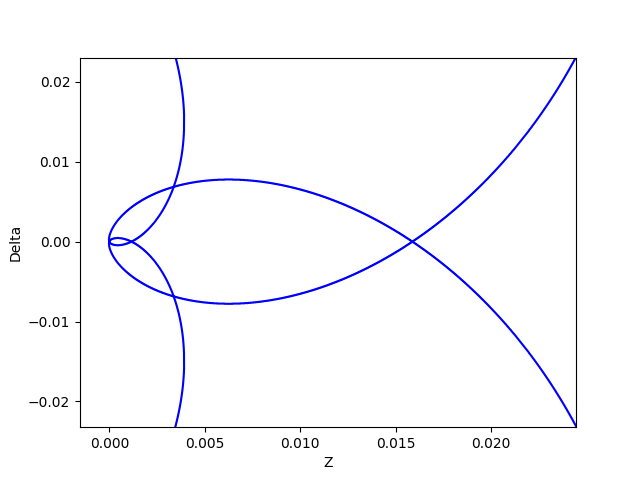}
\put(778,615){$\scriptstyle M_2^-$}
\put(180,615){$\scriptstyle M_1^-$}
\put(187,105){$\scriptstyle M_1^+$}
\put(775,105){$\scriptstyle M_2^+$}
\end{overpic}
    \caption{The $\{ H=0 \}$-slice of the $(2,7)$-system.}
    \label{27Plots}
\end{figure}
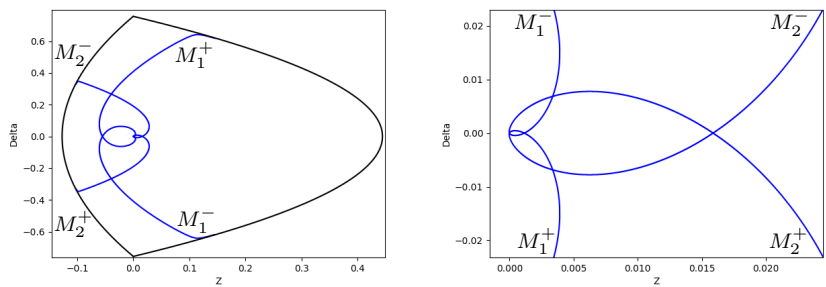

\FloatBarrier


\begin{thebibliography}{Flat}

\bibitem[BB82]{BerardSurNouvellesVariteEinstein}
Lionel B\'{e}rard-Bergery, \emph{Sur de nouvelles vari\'{e}t\'{e}s
  riemanniennes d'{E}instein}, Institut \'{E}lie {C}artan, 6, Inst. \'{E}lie
  Cartan, vol.~6, Univ. Nancy, Nancy, 1982, pp.~1--60.

\bibitem[BGK05]{BoyerGalickiKollarEinsteinMetricsOnSpheres}
Charles~P. Boyer, Krzysztof Galicki, and J\'{a}nos Koll\'{a}r, \emph{Einstein
  metrics on spheres}, Ann. of Math. (2) \textbf{162} (2005), no.~1, 557--580.

\bibitem[BK78]{BourguignonKarcherCurvOperatorsPinchingEstimates}
Jean-Pierre Bourguignon and Hermann Karcher, \emph{Curvature operators:
  pinching estimates and geometric examples}, Ann. Sci. \'{E}cole Norm. Sup.
  (4) \textbf{11} (1978), no.~1, 71--92.

\bibitem[B{\"o}h98]{BohmInhomEinstein}
Christoph B{\"o}hm, \emph{{Inhomogeneous Einstein metrics on low-dimensional
  spheres and other low-dimensional spaces}}, Invent. Math. \textbf{134}
  (1998), no.~1, 145--176.

\bibitem[B{\"{o}}h99]{BohmNonCompactComhomOneEinstein}
Christoph B{\"{o}}hm, \emph{Non-compact cohomogeneity one {E}instein
  manifolds}, Bull. Soc. Math. France \textbf{127} (1999), no.~1, 135--177.

\bibitem[Chi24]{ChiPositiveEinsteinMetrics}
Hanci Chi, \emph{Positive {E}instein metrics with {$\Bbb {S}^{4m+3}$} as the
  principal orbit}, Compos. Math. \textbf{160} (2024), no.~5, 1004--1040.

\bibitem[CL55]{CoddingtonLevinsonODEs}
Earl~A. Coddington and Norman Levinson, \emph{Theory of ordinary differential
  equations}, McGraw-Hill Book Co., Inc., New York-Toronto-London, 1955.

\bibitem[CS19]{CollinsSzekelyhidiSasakiEinsteinMetricsKStability}
Tristan~C. Collins and G\'{a}bor Sz\'{e}kelyhidi, \emph{Sasaki-{E}instein
  metrics and {K}-stability}, Geom. Topol. \textbf{23} (2019), no.~3,
  1339--1413.

\bibitem[DHW13]{DHWShrinkingSolitons}
Andrew~S. Dancer, Stuart~J. Hall, and McKenzie~Y. Wang, \emph{Cohomogeneity one
  shrinking {R}icci solitons: an analytic and numerical study}, Asian J. Math.
  \textbf{17} (2013), no.~1, 33--61.

\bibitem[EW00]{EschenburgWangInitialValueProblem}
J.-H. Eschenburg and McKenzie~Y. Wang, \emph{The initial value problem for
  cohomogeneity one {E}instein metrics}, J. Geom. Anal. \textbf{10} (2000),
  no.~1, 109--137.

\bibitem[FH17]{FoscoloHaskinsNearlyKaehler}
Lorenzo Foscolo and Mark Haskins, \emph{New {$G_2$}-holonomy cones and exotic
  nearly {K}\"{a}hler structures on {$S^6$} and {$S^3\times S^3$}}, Ann. of
  Math. (2) \textbf{185} (2017), no.~1, 59--130.

\bibitem[GK07]{GhigiKollarKaehlerEinsteinOrbifolds}
Alessandro Ghigi and J\'{a}nos Koll\'{a}r, \emph{K\"{a}hler-{E}instein metrics
  on orbifolds and {E}instein metrics on spheres}, Comment. Math. Helv.
  \textbf{82} (2007), no.~4, 877--902.

\bibitem[Jen73]{JensenEinsteinMetricsOnPrincipalFibreBundles}
Gary~R. Jensen, \emph{Einstein metrics on principal fibre bundles}, J.
  Differential Geometry \textbf{8} (1973), 599--614.

\bibitem[LST24]{LiuSanoTasinIninitelySasakiEinsteinMetricsSpheres}
Yuchen Liu, Taro Sano, and Luca Tasin, \emph{{Infinitely many families of
  Sasaki-Einstein metrics on spheres}}, to appear in J. Differential Geometry,
  arXiv:2203.08468 (2024).

\bibitem[Win23]{WinkSolitonsFromHopfFibrations}
Matthias Wink, \emph{{Cohomogeneity one Ricci Solitons from Hopf Fibrations}},
  Comm. Anal. Geom. \textbf{31} (2023), no.~3, 625--676.

\bibitem[Zil82]{ZillerHomogeneousEinsteinMetrics}
W.~Ziller, \emph{Homogeneous {E}instein metrics on spheres and projective
  spaces}, Math. Ann. \textbf{259} (1982), no.~3, 351--358.
  
\end{thebibliography}
\end{document}